\documentclass[preprint]{elsarticle}

\biboptions{longnamesfirst,semicolon}
\usepackage[T1]{fontenc}
\usepackage[utf8]{inputenc}
\usepackage{amsmath,amssymb}
\usepackage{amsthm}
\usepackage[english]{babel}
\usepackage{color}

\numberwithin{equation}{section}

\usepackage{mathrsfs}

\newtheorem{theorem}{Theorem}
\newtheorem{proposition}{Proposition}
\newtheorem{remark}{Remark}
\newtheorem{lemma}{Lemma}
\newtheorem{corollary}{Corollary}
\newtheorem{definition}{Definition}

\newcommand{\s}{\sigma}
\newcommand{\p}{\psi }
\newcommand{\pps}{\phi_{\p }}

\newcommand{\op}{\overline{\p }}
\newcommand{\lp }{\left(}
\newcommand{\rp  }{\right)}
\newcommand{\lc }{\left[ }
\newcommand{\rc }{\right]}

\newcommand{\D}{\mathcal{D} }
\newcommand{\R}{\mathbb{R}}
\newcommand{\T}{\mathbb{T}}
\newcommand{\e}{{\rm e}}
\newcommand{\pr}{\partial}
\newcommand{\prx}{\partial_{x_1}}
\newcommand{\tp}{\tilde{\p}}

\newcommand{\MM}{\mathfrak{M}}

\begin{document}

\title{Local and global solution for a nonlocal Fokker-Planck equation related to the adaptive biaising force processes}

\author{Houssam Alrachid} \ead{houssam.alrachid@enpc.fr}
\address{\'Ecole des Ponts ParisTech, Université Paris Est, 6-8 Avenue Blaise Pascal, Cité Descartes \\
Marne-la-Vallée,  F-77455 ,
France\\}
\address{Université Libanaise, Ecole Doctorale des Sciences et de Technologie\\
 Beyrouth, Lebanon}
 
\author{Tony Lelièvre}
\ead{lelievre@cermics.enpc.fr}
\address{\'Ecole des Ponts ParisTech, Université Paris Est, 6-8 Avenue Blaise Pascal, Cité Descartes \\
Marne-la-Vallée,  F-77455 ,
France\\}

\author{Raafat Talhouk} \ead{rtalhouk@ul.edu.lb}
\address{Université Libanaise, Faculté des Sciences et Ecole Doctorale des Sciences et de Technologie\\
 Beyrouth, Liban}

\begin{abstract}

We prove global existence, uniqueness and regularity of the mild, $L^p$ and classical solution of a non-linear {\em Fokker-Planck} equation arising in an adaptive importance sampling method for molecular dynamics calculations. The non-linear term is related to a conditional expectation, and is thus non-local. The proof uses tools from the theory of semigroups of linear operators for the local existence result, and an a priori estimate based on a supersolution for the global existence result. 

\end{abstract}


\maketitle

  \section{Introduction}\label{intro}

We consider the following {\em Fokker-Planck} equation 
  \begin{equation}\label{fpp}
    \left \lbrace
  \begin{aligned}
\pr_t\p &={\rm div}\big( \nabla V\p + \beta^{-1}\nabla\p \big)-\prx (\pps \p)&\mbox{in }(0,\infty )\times  \T^n, \\
\p(.,0)&=\p_0&\mbox{in } \T^n,
  \end{aligned}
  \right.
\end{equation}
with periodic boundary conditions on the unit torus $\T^n$ of dimension $n\geq 2$, where $\T=\R /\mathbb{Z}$ denotes the one-dimensional unit torus. We assume $\p_0\in W^{\s,p}(\T^n)$, $p>n$, with $\p_0\geq 0$ and $\displaystyle \int_{\T^n} \p_0=1$, $0<\s<2$ and $p>n$ to be fixed later on. The function $V: \T^n \rightarrow \R $ denotes the potential energy assumed to be a $C^2$ function and $\beta$ is a positive constant proportional to the inverse of the temperature $T$. The function $\p\mapsto \phi_\p$ is defined from $W^{1,p}(\T^n)$ into $W^{1,p}(\T^n)$ as follows 
  \begin{equation}\label{fi}
     \pps (t,x_1)=\frac{\displaystyle \int_{\T^{n-1}} \prx  V(x)\p (t,x) dx_2...dx_n}{\op (t,x_1)},
 \end{equation}
 where
 \begin{equation}
\op (t,x_1)= \displaystyle \int_{\T^{n-1}}\p (t,x) dx_2...dx_n.
 \end{equation}
Notice that $\pps $ is well defined if $\op (t,x_1)\neq 0$, $\forall x_1 \in \T$.
 Therefore, we will work on the following open subset of $W^{\s,p}(\T^n)$:
\begin{equation}\label{D}
\D ^{\s,p}(\T^n):=\{\p \in W^{\s,p}(\T^n)\,|\,\op >0\}.
\end{equation}

The partial differential equation \eqref{fpp} is a parabolic equation with a nonlocal nonlinearity. A solution of the Fokker-Planck equation is a probability density function. The parabolic system \eqref{fpp} can be rewritten as
\begin{equation}\label{fp2}
\left\{
\begin{aligned}
\dot{\p } -\beta ^{-1}\Delta\p & =F  (\p) \mbox{ in}\,  (0,\infty ),\\
\p(0) &=\p_0,
\end{aligned}
\right.
\end{equation}
where $\dot{\p }=\frac{d\p}{dt}$ denotes the time derivative and 
$$ F  (\p):=\nabla V . \nabla \p + \Delta V \p -\prx (\pps \p).$$

Such Fokker-Planck problems (i.e \eqref{loc}) arise in adaptive methods for free energy computation techniques. Many molecular dynamics computations aim at computing free energy, which is a coarse-grained description of a high-dimensional complex physical system (see \citep{chipot2007free,tony:10}). More precisely, \eqref{loc} rules the evolution of the density (i.e. $\p(t)$) of a stochastic process $X(t)$ that is following an adaptively biased overdamped Langevin dynamics called {\em ABF} (or Adaptive biasing force method). The nonlinear and nonlocal term $\pps $, defined in \eqref{fi}, is used during the simulation in order to remove the metastable features of the original overdamped Langevin dynamics (see \citep{al:15,lel:07} for more details).
\medskip

Up to our knowledge, this is the first time that parabolic problems with nonlinearities involving the nonlocal term \eqref{fi} are studied. Different types of nonlocal nonlinearities have been studied in \citep{qui:07} for instance. A proof of existence of a solution to \eqref{fpp} is also obtained in \citep{jourdain2010existence} using probabilistic arguments. Here, we use analytical techniques that we expect to be more robust to extend the result to more general settings. 

\medskip

Before we present our main results, we define the mild, the $L^p$ and the maximal solutions of the parabolic problem \eqref{fpp}.
\begin{definition}(Mild, $L^p$ and maximal solution)\label{def}\\
Suppose that $0<\s\leq 2$ and $p>n$. Let $\p:[0,T)\rightarrow L^p(\T^n)$, where $0<T\leq\infty$ and $\p(0)=\p_0\in \D^{\s,p}(\T^n)$.  
\begin{itemize}
\item[$(i)$]  $\p$ is said to be a \textit{mild solution } of \eqref{fp2}, if $\p\in C([0,T),\D^{\s,p}(\T^n))$ satisfies the following integral-evolution equation:
\begin{equation}\label{iee}
\p (t)=\mathrm{e}^{\beta^{-1}t\Delta }\p_0+\displaystyle \int_{0}^{t}\mathrm{e}^{\beta^{-1}(t-s)\Delta } F  (\p (s))ds,\quad t\in [0,T);
\end{equation} 
\item[$(ii)$] $\p$ is said to be a \textit{$L^p-$solution} of \eqref{fpp} on $[0,T)$, if $\p\in C([0,T),L^p(\T^n))\cap C^1((0,T),L^p(\T^n))$, $\p(t)\in W^{2,p}(\T^n)\cap \D^{\s,p}(\T^n)$, for $t\in (0,T)$ and $\dot{\p }(t) -\beta^{-1}\Delta \p(t) = F  (\p(t))$ in $L^p(\T^n)$, for $t\in (0,T) $ and $\p(0)=\p_0$.
\item[$(iii)$] $\p$ is a maximal mild (resp. $L^p-$) solution if there does not exist a mild (resp. $L^p-$) solution of \eqref{fpp} which is a proper extension of $\p$. In this case, its interval of definition in time $(0,T_{\text{max}}):=(0,T)$ is said to be a \textit{maximal interval}.
\end{itemize}
\end{definition}
As will become clear below, all the definitions make sense since $F$ is well defined from $\D^{\s,p}$ into $L^p(\T^n)$, thanks to the assumption on $\s$ and $p$.
\medskip

In this paper, we will use the following hypothesis:
\begin{equation}\label{h}
[\mathcal{H}_1]\quad\p_0\in\D^{\s,p}(\T^n), \p_0\geq 0  \text{ and } \displaystyle\int_{\T^n}\p_0=1 .
\end{equation}
\begin{equation}\label{h2}
[\mathcal{H}_2]\!\quad n\geq 2,\,p>n \text{ and }\s\in(1+n/p,2).
\end{equation}

Our first main result concerning local-in-time existence and regularity is the following theorem.

\begin{theorem}\label{loc}
Assume $[\mathcal{H}_1]$ and $[\mathcal{H}_2]$. The initial boundary value problem \eqref{fpp} has a unique maximal $L^p$-solution $\p(t)$ with maximal interval of existence $(0,T_{\text{max}})$, where $T_{\text{max}}:=T_{\text{max}}(\p_0)>0$. 

Moreover,
\begin{itemize}
\item[(i)]$\forall \varepsilon \in [0,1-\s/2) $, $\forall \tau \in [0,\s)$, $\p \in  C((0,T_{\text{max}}),W^{2,p}(\T^n)) \cap C^\mu_{\text{loc}}([0,T_{\text{max}}),$ $W^{\tau,p}(\T^n))$, with
 $\mu:=\min\lp 1-\frac{\s}{2}-\varepsilon,\frac{\s-\tau}{2} \rp $;
\item[(ii)] $\displaystyle \frac{d\p}{dt}\in C^{\nu^-}_{\text{loc}}((0,T_{\text{max}}),L^p(\T^n))$, where $\nu:=\min ( 1-\frac{\s}{2},\frac{\s}{2}-\frac{1}{2}-\frac{n}{2p})$. By $C^{\nu^-}$, we mean $C^{\nu-\varepsilon}$, for any $\varepsilon\in (0,\nu)$; 
\item[(iii)] For $\p_0\in \D ^{\s_{\text{opt}},p}(\T^n)$, where $\s_{\text{opt}}=\frac{1}{5}\lp8+\frac{2n}{p}\rp$, then $\forall \rho\in (0,\frac{1}{5}(1-\frac{n}{p})) $, $\p \in C^{\rho}([0,T_{\text{max}}),C^{1+\rho}(\T^n))$. In the case when $\p_0\in \D^{\frac{8}{5},\infty}(\T^n)$, then $\forall \rho\in (0,\frac{1}{5}) $, $\p \in C^{\rho}([0,T_{\text{max}}),C^{1+\rho}(\T^n))$;
\item[(iv)] $\p$ is a classical solution, which means that $\p $ belongs to $C^{1}( (0,T_{\text{max}}),C^2(\T^n))$.
\end{itemize}
\end{theorem}

The proofs of local existence are inspired from~\citep{am:84} and~\citep{paz:83}. The existence of the unique local-in-time solution relies on the fact that $F$ is locally Lipschitz continuous from $\D^{\s,p}(\T^n)$ into $L^p(\T^n)$ (see Lemma~\ref{lips}). Another fondamental ingredient is the following proposition, which will be used intensively throughout this paper.

  \begin{proposition}\label{diff}
Assume that $\p$ is a $L^p$-solution of \eqref{fpp}, then $\op(t,x_1)=\displaystyle\int_{\T^{n-1}}\p (t,x) dx_2...dx_n$ is the unique solution in the distribution sense of the following diffusion equation:
\begin{equation}\label{diffeq}
\left\{
\begin{aligned}
\pr_t\op &=\beta^{-1}\pr_{x_1x_1}^2\op \text{ in } [0,\infty)\times\T ,\\
\op (0,.)&=\op _0 \mbox{ on } \T.\\
\end{aligned}
\right.
\end{equation}
\end{proposition}
\begin{remark}\label{posdi}
 Since, by Proposition \ref{diff}, $\op$ satisfies a simple diffusion equation, then the property $ \op > 0$ is propagated in time. Moreover, since $ \op_0\geq 0$ and $\int_{\T}\op_0=1$, then up to considering the problem for $t\geq t_0>0$, one can assume that $\op_0>0$. This will be assumed in the following (see Definition~\ref{def}) . In addition, it follows from known results on parabolic linear equation (see for example \citep{ev:98} or \citep{lib:96}) that $\op\in C^\infty((0,\infty),C^\infty(\T^n)) $.
 \end{remark}

We will check that $\p$ is a probability density function. In particular, the positivity of $\p$ can be verified upon some positivity conditions on $\p_0$ as following. 
\begin{proposition}\label{pos1}
Assume $[\mathcal{H}_1]$ and $[\mathcal{H}_2]$. Then the $L^p-$solution $\p(t)$ of \eqref{fpp} satisfies $\p(t)> 0$ for all $t\in (0,T_{\text{max}}).$
\end{proposition}

The local-in-time existence result is expected and rather standard since $F$ is locally Lipschitz continuous. The main difficulty of this work is then to obtain a-priori estimates to prove the global-in-time existence and uniqueness. This is done by exhibiting a supersolution of the partial differential equations satisfied by $\p\exp(\frac{\beta V}{2})$, which only depends on $t$ and $x_1$ (see Section~\ref{pub}).
\medskip

The second main result states the global existence of the solution to \eqref{fpp}. 

\begin{theorem}\label{glo}
Assume $[\mathcal{H}_1]$ and $[\mathcal{H}_2]$. Let $\p$ be the solution of \eqref{fpp} given by Theorem~\ref{loc}, with maximal interval of existence $(0,T_{\text{max}})$. Then
\begin{itemize}
\item[(i)] $T_{\text{max}}=+\infty$;
\item[(ii)] For every $\delta > 0$, $\displaystyle \sup_{t\geq\delta}\|\p(t)\|_{\rho,p}<\infty $ for every $\rho\in [\s, 2)$ and the orbit set $\gamma^{+}(\p_0):=\{\p(t,\cdot) ;0 \leq t< T_{\text{max}}\}$ is relatively compact in $C^1(\T^n)$.
\end{itemize}
\end{theorem}
\medskip

Let us make a few comments on the functional framework we use. In this paper, we work in $L^p(\T^n)$ with $p>n$ to ensure that $W^{\s,p}(\T^n)$ is an algebra for $\s\geq 1$. In addition, the parameter $\s$ is restricted to the interval $(1+\frac{n}{p},2)$ since we need on the one hand the Sobolev embedding $W^{\s,p}(\T^n)\hookrightarrow C^1(\T^n) $ (see \eqref{boun} below) and, on the other hand, $t\mapsto \displaystyle \|\e^{-tA_p}\|_{\mathcal{L}(L^p_0,W^{\s,p}_0)}\in L^1(0,+\infty) $, which requires $\s <2$ (see \eqref{or1} below).
\medskip

 The paper is organized as follows. In Section~\ref{not} we provide some notations and preliminaries. The Section~\ref{theo} contains the proof of Theorem~\ref{loc}, which states the local-in-time existence and uniqueness of solution to \eqref{fpp}. In Section~\ref{max}, we prove Proposition~\ref{diff} and we use a weak maximum principle to prove that $\p\geq 0$. In Section~\ref{priob}, $L^p$-bounds for the nonlinear functional $F$ and a-priori bounds of $\p$ are proved, which yield the global in time existence theorem (Theorem~\ref{glo}). 

\section{Notations and preliminaries}\label{not}
We denote by $L^p(\T^n)$ and $W^{\s,p}(\T^n)$ the usual Lebesgue and Sobolev spaces. The space $C^{\alpha + k}(\T^n)$, with $k\in \mathbb{N}$ and $\alpha\in (0,1)$, is the Banach space of all functions belonging to $C^{ k}(\T^n)$ whose $k^{th}$ order partial derivatives are $\alpha -$H\"{o}lder continuous on $\T^n$. We denote by $\|.\|_{p}$ and $\|.\|_{s,p}$ the usual norms on $L^p(\T^n)$ and $W^{s,p}(\T^n)$ respectively. The norm on $C^{\alpha + k}(\T^n)$ is defined by:
$$\|f\|_{C^{\alpha + k}}:=\displaystyle \max_{|i|\leq k}\sup_{x\in\T^n}|D^if(x)|+\displaystyle \max_{|i|=k}\sup_{x\neq y}\frac{|D^if(x)-D^if(y)|}{|x-y|^\alpha}. $$ 

\medskip

$C^{1^-}(\T^n)$ (resp. $C^{1^-}_{\text{loc}}(\T^n)$) is the space of globally (resp. locally) Lipschitz continuous functions on $\T^n$. We mean by $f \in C^{0,1-}((0,T)\times W^{\s,p},L^p)$, that the mapping $f(\cdot,\p):(0,T)\rightarrow L^p$ is continuous for each $\p\in W^{\s,p}$ and $f(t,.):W^{\s,p}\rightarrow L^p$ is uniformly Lipschitz continuous for each $t\in (0,T)$.

\medskip

$B_{\s,p}(0,R)$ denotes the ball of radius $R>0$, in the topology of the space $W^{\s,p}(\T^n)$. For $T>0$ and $\rho\in (0,1)$, the space $C^{\rho}((0,T),L^p(\T^n))$ (resp. $C_{\text{loc}}^{\rho}((0,T),L^p(\T^n))$) denotes the space of globally (resp. locally) $\rho-$H\"older continuous functions from $(0,T)$ to $L^p(\T^n)$. Recall that a H\"older continuous function on a compact is equivalently a locally H\"older continuous function on this compact.

\medskip

By $e_1$ we denote the unit vector $(1,0,\cdot\cdot\cdot,0)$ of $\R^n$ and $C$ denotes various positive constants which may vary from step to step.

\medskip

We shall use the following Sobolev embeddings, that can be found in \citep{bre,gil:77,ada:03}:
\begin{equation}
 W^{s,p}(\T^n) \hookrightarrow C^m(\T^n),\text{ if } 0\leq m<s-n/p , \label{boun} 
\end{equation}
\begin{equation}
 W^{k,q}(\T^n) \hookrightarrow W^{l,p}(\T^n),\text{ if }k>l,\text{ and }k-\frac{n}{q}>l-\frac{n}{p}.\label{sobo} 
\end{equation}

In the following, we will use the operator $A_p=-\beta^{-1}\Delta$. The domain of $ A_p $ is $\displaystyle D( A_p )=\displaystyle W^{2,p}(\T^n)\cap L^p_0(\T^n)$, where $L^p_0(\T^n):=\{\p\in L^p(\T^n)\,|\,\displaystyle\int_{\T^n}\p=0\} $, equipped with the $L^p$-norm. The operator $- A_p $ generates a strongly continuous analytic semigroup of contraction $\{e^{-t A_p };\, t\geq 0\}$ on $L^p_0(\T^n)$ (see Lemma~\ref{spa}).. The domain of $ A_p $ is $\displaystyle D( A_p )=\displaystyle W^{2,p}(\T^n)\cap L^p_0(\T^n)$. Then the mild solution we be written in terms of $A_p$ rather than $-\beta^{-1}\Delta$ such that a mild solution will actually be solution to (compare with \eqref{iee}):
\begin{equation}\label{mld2}
\p(t)-1=\mathrm{e}^{-t A_p }(\p _0-1)+\displaystyle \int_{0}^{t}\mathrm{e}^{-(t-s) A_p }F (\p(s))ds,\quad t\in [0,T).
\end{equation}
Indeed, we have that $\displaystyle\int_{\T^n}(\p-1)=0 $, $\displaystyle\int_{\T^n}(\p_0-1)=0 $ and $\displaystyle\int_{\T^n}F(\p)=0$ (since $F$ is periodic and in divergence form). In addition, since $\e^{t\beta^{-1}\Delta}1=1$, a solution $\p$ of \eqref{mld2} gives a solution $\p$ of \eqref{iee} satisfying $\displaystyle\int_{\T^n}\p=1$.

In the following, we will use the notation $A_p$ and $\e^{-tA_p}$, when the operator applies to a function in $\D(A_p)$ (as in \eqref{mld2}) and the notation $-\beta^{-1}\Delta$ and $\e^{t\beta^{-1}\Delta}$ otherwise.
\medskip

The following Lemmas will be used in the next sections. These are classical results that we recall in our specific context, for the sake of consistency (see \citep{paz:83} or \ref{prspa} and~\ref{prspa2} for proofs).

\begin{lemma}\label{spa}
Let $1<p<\infty$, then the operator $ A_p =-\beta^{-1}\Delta$, with domain $D(A_p)=W^{2,p}(\T^n)\bigcap L^p_0(\T^n) $, satisfies the following assertions:
\begin{enumerate}
\item The operator $- A_p $ generates a strongly continuous analytic semigroup of contraction $\{e^{-t A_p };\, t\geq 0\}$ on $L^p_0(\T^n)$. In particular $\|e^{-t A_p }\|_{\mathcal{L}(L^p_0,L^p_0)}\leq 1$.
\item The spectrum of $ A_p $ is included in  $\R_+^*$.
\item  $\forall \alpha\geq 0,\exists C_{\alpha}>0,\exists \kappa>0$ such that
 \begin{equation}\label{reg1}
\forall t>0,\quad \left\| A_p ^{\alpha}\e^{-t A_p }\right\|_{\mathcal{L}(L^p_0,L^p_0)}\leq C_\alpha t^{-\alpha}\e^{-\kappa t}. 
 \end{equation}
 \item  $ \forall \alpha\in  [0,1],\exists C_{\alpha}>0, \forall \p_0\in D( A_p ^\alpha)$ 
 \begin{equation}\label{reg2}
\forall t>0,\quad\left\|(\e^{-t A_p }-I)\p_0\right\|_p\leq C_\alpha t^{\alpha}\| A_p ^\alpha \p_0\|_p. 
 \end{equation}
\item $\forall \s\in [0,2)$, $\exists\, C_{\s}>0$, $\exists \kappa>0$ such that
\begin{equation}\label{or1}
\forall t>0,\quad \|\mathrm{e}^{-tA_p }\|_{\mathcal{L}(L^p_0,W^{\sigma,p}_0)}\leq \widehat{\alpha}(t),
\end{equation}
where $\widehat{\alpha}(t)=C_\s t^{-\frac{\sigma}{2}}e^{-\kappa t}$. Note that $\widehat{\alpha}$ is in $ L^1((0,\infty ),\mathbb{R}_+)$ and is a decreasing function.\label{e1}
\item $\forall \s\in [0,2]$, $\forall \gamma\in [\s/2,1] $, $\exists\, C_{\s,\gamma}>0$, $\exists \kappa>0$ such that $\forall\, 0\leq s<r<t$
\begin{equation}\label{dex}
\|\e^{-(t-s) A_p }-\e^{-(r-s) A_p }\|_{\mathcal{L}(L^p_0,W^{\sigma,p}_0)}\leq C_{\s,\gamma}(t-r)^{\gamma-\s/2}(r-s)^{-\gamma}\e^{-\kappa (r-s)}.
\end{equation}

\end{enumerate}
\end{lemma}

\begin{lemma}\label{spa2}
Let $\p_0\in L^p_0(\T^n)$ and $\{e^{-t A_p };\, t\geq 0\}$ be the continuous analytic semigroup of contraction defined in the previous lemma, then
\begin{enumerate}
\item \begin{equation}\label{int1}
  \displaystyle\forall t\geq 0, \lim_{h\rightarrow 0}\int_t^{t+h}\e^{-s A_p }\p_0  ds=\e^{-t A_p }\p_0  .
 \end{equation}
 \item The integral $\displaystyle\int_0^{t}\e^{-s A_p }\p_0  ds$ is in $ D( A_p )$ and
 \begin{equation}\label{int2}
 \displaystyle\forall t\geq 0,- A_p \lp \int_0^t\e^{-s A_p }\p_0  ds\rp =\e^{-t A_p }\p_0  -\p_0  .
 \end{equation}
 \item For $\p_0  \in D( A_p )$,
 \begin{equation}\label{int3}
\forall 0<s<t, \e^{-t A_p }\p_0  -\e^{-s A_p }\p_0  =\displaystyle\int_s^t- A_p \e^{-\tau A_p }\p_0  d\tau.
 \end{equation}

\end{enumerate}
\end{lemma}

\section{Local existence}\label{theo}
This section is devoted to a proof of the local existence of solution to the partial differential equation \eqref{fpp}. In Section~\ref{exs}, we show the existence of mild solution. Section~\ref{reg} is devoted to some regularity results for the mild solution. Finally, we prove Theorem~\ref{loc} in Section~\ref{prth}.

\subsection{Existence of mild solution}\label{exs}
In this section, we show that $ F  :\D ^{\s,p}(\T^n)\rightarrow L^p_0(\T^n)$ is a locally Lipschitz continuous function, which is essential to prove the existence of a mild solution by using the Banach fixed point theorem.

\begin{lemma}\label{finft}
Let $\p \in C(\T^n)$ and suppose that $\op (t,x_1):= \displaystyle \int_{\T^{n-1}}\p (t,x) dx_2...dx_n>0$. Then  
\begin{equation}
\|\pps \|_{\infty}\leq \|\prx V\|_{\infty}.
\end{equation}
Moreover, for all $\p \in C^1(\T^n)$, we have that 
\begin{equation}\label{cc1}
\|\prx \pps \|_{\infty}\leq \frac{C}{\min \op}\|\p\|_{C^1},
\end{equation}
where $C$ depends only on the potential $V$.
\end{lemma}
\begin{proof} 
The first assertion is easy to prove. For the second assertion, since $\op >0$ and $\p\in C(\T^n)$, then there exists a constant $\alpha>0$, such that $\op (t,x_1)\geq \alpha>0$, for all $x_1\in\T$. Therefore,
\begin{equation}\label{der}
\displaystyle\prx \pps =\displaystyle\frac{\int_{\T^{n-1}}\pr_{x_1x_1}^2V\p}{\int_{\T^{n-1}}\p}+\frac{\int_{\T^{n-1}}\prx V\prx \p}{\int_{\T^{n-1}}\p} -\frac{\int_{\T^{n-1}}\prx V\p}{\int_{\T^{n-1}}\p}\frac{\int_{\T^{n-1}}\prx \p}{\int_{\T^{n-1}}\p}.
\end{equation}
Using \eqref{der}, one obtains
$$\begin{aligned}
\|\prx \pps \|_{\infty}&\leq\frac{1}{\alpha }\|\pr_{x_1x_1}^2V\|_{\infty}\|\p\|_{\infty}+\frac{1}{\alpha }\|\prx V\|_{\infty}\|\prx \p\|_{\infty}\\
&\quad+\frac{1}{\alpha }\|\prx V\|_{\infty}\|\prx \p\|_{\infty},
\end{aligned}$$
which yields \eqref{cc1}.
\end{proof}
\begin{lemma}\label{majfi}
Assume $[\mathcal{H}_2]$ (defined in \eqref{h2}). For all $\p_1$ and $\p_2$ $\in \D ^{\s ,p}(\T^n)$, there exists $C^1_{\p_1,\p_2}$ and $C^2_{\p_1,\p_2}$ such that
\begin{align}
\|\phi_{\p_1}-\phi_{\p_2}\|_p&\leq C^1_{\p_1,\p_2} \|\p_1-\p_2\|_p\label{phi1}\\
\|\prx (\phi_{\p_1}-\phi_{\p_2})\|_p&\leq C^2_{\p_1,\p_2}\|\p_1-\p_2\|_{1,p}\label{phi2}
\end{align}

where $$C^1_{\p_1,\p_2}:=C\frac{\min (\|\p_1\|_{1,p},\|\p_2\|_{1,p})}{\min\,\op _1\min\,\op _2}$$
 and 
  $$C^2_{\p_1,\p_2}:=C\frac{\min(\|\p_1\|_{\s,p},\|\p_2\|_{\s,p})}{\min\,\op _1\min\,\op _2}+C\frac{\|\p_1\|_{\s,p}\|\p_2\|_{\s,p}(\|\p_1\|_{\s,p}+\|\p_2\|_{\s,p})}{(\min\,\op _1)^2(\min\,\op _2)^2} .$$ 
The $W^{\s,p}-$norms are finite since $\p_1$ and $\p_2 \in\D ^{\s ,p}(\T^n)$. The constant $C$ depends only on the potential $V$.
\end{lemma}
\begin{proof}
Let $\p_1,\,\p_2\,  \in \D ^{\s ,p}(\T^n)$. There exists $\alpha_1>0$ and $\alpha_2>0$ such that $\op _1\geq \alpha_1$ and $\op _2\geq \alpha_2$. Now,

$\begin{aligned}
\|\phi_{\p_1}-\phi_{\p_2}\|_p&\leq \frac{1}{\alpha_1\alpha_2}\left \|\displaystyle\op_2 \int_{\T^{n-1}}\prx V\p_1+\op_1 \int_{\T^{n-1}}\prx V\p_2\right\|_p\\
&\leq \frac{1}{\alpha_1\alpha_2}\left \|(\op _2-\op _1)\!\!\!\displaystyle \int_{\T^{n-1}}\!\!\!\!\!\!\prx V\p_1\right\|_p\!\!\!+\frac{1}{\alpha_1\alpha_2}\left \|\op _1\!\!\!\displaystyle \int_{\T^{n-1}}\!\!\!\!\!\!\prx V(\p_2-\p_1)\right\|_p\!\!\!.
\end{aligned}$\\

Using the embedding \eqref{boun},\\

$\begin{aligned}
\|\phi_{\p_1}-\phi_{\p_2}\|_p&\leq  \frac{1}{\alpha_1\alpha_2}\|\prx V\|_{\infty}\|\p_1\|_{\infty}\|\op _1-\op _2\|_p\\
&\quad +\frac{1}{\alpha_1\alpha_2}\|\prx V\|_{\infty}\|\op _1\|_{\infty}\|\p_1-\p_2\|_p\\
&\leq \frac{1}{\alpha_1\alpha_2}\|\prx V\|_{\infty}\|\p_1\|_{\infty}\left\|\int_{\T^{n-1}}(\p_1-\p_2)\right\|_p\\
&\quad +\frac{1}{\alpha_1\alpha_2}\|\prx V\|_{\infty}\|\p_1\|_{\infty}\|\p_1-\p_2\|_p\\
&\leq\frac{1}{\alpha_1\alpha_2}\|\prx V\|_{\infty}\|\p_1\|_{\infty}\|\p_1-\p_2\|_{p}\\
&\quad +\frac{1}{\alpha_1\alpha_2}\|\prx V\|_{\infty}\|\p_1\|_{\infty}\|\p_1-\p_2\|_p\\
&\leq \frac{C\|\p_1\|_\infty}{\alpha_1\alpha_2}\|\p_1-\p_2\|_p.
\end{aligned}$\\

Since the left hand side is symmetric in $(\p_1,\p_2)$, one can take the minimum of the upper bounds obtained by permutation of $(\p_1,\p_2)$. This concludes the proof of \eqref{phi1}.
\medskip

 A similar analysis can be done for the proof of the second assertion:
 
$\begin{aligned}
\prx (\phi_{\p_1}-\phi_{\p_2})&= \frac{\displaystyle \int_{\T^{n-1}}\prx (\prx V\p_1)}{\op _1}-\frac{\left [\displaystyle\int_{\T^{n-1}}\prx V\p_1 \rc \prx \op _1}{(\op _1)^2}\\
&\quad - \frac{\displaystyle \int_{\T^{n-1}}\prx (\prx V\p_2)}{\op _2}+\frac{\left [\displaystyle\int_{\T^{n-1}}\prx V\p_2 \rc \prx \op _2}{(\op _2)^2}\\
&=\frac{\displaystyle\op _2\int_{\T^{n-1}}\prx (\prx V\p_1)-\op _1\int_{\T^{n-1}}\prx (\prx V\p_2)}{\op _1\op _2}\\
&\quad - \frac{(\displaystyle\op _2)^2\lc \int_{\T^{n-1}}\prx V\p_1\rc \prx \op _1
-(\op _1)^2\lc \int_{\T^{n-1}}\prx V\p_2\rc \prx \op _2}{(\op _1\op _2)^2}.
\end{aligned}$

Using the embeddings $W^{1,p}(\T^n)\hookrightarrow L^{\infty}(\T^n) $ and $W^{\s,p}(\T^n)\hookrightarrow W^{1,\infty}(\T^n) $ (see~\eqref{boun}), then\\

$\begin{aligned}
&\left\|\op _2\int_{\T^{n-1}}\prx (\prx V\p_1)-\op _1\int_{\T^{n-1}}\prx (\prx V\p_2)\right\|_p\\
&\leq\left\|(\op_2-\op_1)\int_{\T^{n-1}}\prx (\prx V\p_1)\right\|_p+\left\|\op _1\int_{\T^{n-1}}\prx (\prx V(\p_1-\p_2)\right\|_p\\
&\leq \left\|\op_2-\op_1\right\|_p\left\|\prx (\prx V\p_1)\right\|_{\infty} +\|\op_1\|_{\infty}\left\|\prx (\prx V(\p_1-\p_2)\right\|_p\\
&\leq \left\|\p_2-\p_1\right\|_p\lp \|\pr_{x_1x_1}V\p_1\|_{\infty}+\|\prx V\prx \p_1\|_{\infty}\rp\\
&\quad  +\|\p_1\|_{1,p}\lp \|\pr_{x_1x_1}V(\p_1-\p_2)\|_p+\|\prx V\prx (\p_1-\p_2)\|_{p}\rp \\
&\leq C\left\|\p_1-\p_2\right\|_p\|\p_1\|_{1,p}+C\left\|\p_1-\p_2\right\|_p\|\p_1\|_{\s,p}+C\|\p_1\|_{1,p}\left\|\p_1-\p_2\right\|_{1,p}\\
&\leq C\|\p_1\|_{\s,p}\left\|\p_1-\p_2\right\|_{1,p}.
\end{aligned}$\\

In addition,\\

$\begin{aligned}
&(\displaystyle\op _2)^2\lc \int_{\T^{n-1}}\prx V\p_1\rc \prx \op _1
-(\op _1)^2\lc \int_{\T^{n-1}}\prx V\p_2\rc \prx \op _2\\
&=(\op _2)^2\prx \op_1\int_{\T^{n-1}}\prx V(\p_1-\p_2)+\lp (\op_2)^2-(\op_1)^2\rp \prx \op_1\int_{\T^{n-1}}\prx V\p_2\\
&\quad +(\op_1)^2\prx (\op_1-\op_2)\int_{\T^{n-1}}\prx V\p_2,
\end{aligned}$

thus,\\

$\begin{aligned}
&\left\|(\displaystyle\op _2)^2\lc \int_{\T^{n-1}}\prx V\p_1\rc \prx \op _1
-(\op _1)^2\lc \int_{\T^{n-1}}\prx V\p_2\rc \prx \op _2\right\|_p\\
&\leq \|\prx V\|_{\infty}\|\op _2\|^2_{\infty}\|\prx \op_1\|_{\infty}\|\p_1-\p_2\|_p  \\
&\quad+\|\prx V\|_{\infty}\|\op_1-\op_2\|_p \|\op_1+\op_2\|_{\infty}\|\prx \op_1\|_{\infty}\|\p_2\|_{\infty}\\
&\quad +\|\prx V\|_{\infty}\|\op_1\|_{\infty}^2\|\p_2\|_{\infty}\|\op_1-\op_2\|_{1,p}\\
&\leq \|\prx V\|_{\infty}( \|\p _1\|_{\s,p}\|\p _2\|^2_{1,p}+\|\p_1+\p_2\|_{1,p}\|\p_1\|_{\s,p}\|\p_2\|_{1,p}\\
&\quad +\|\p_1\|_{1,p}^2\|\p_2\|_{1,p}) \|\p_1-\p_2\|_{1,p}\\
&\leq C\|\p_1\|_{\s,p}\|\p_2\|_{\s,p}(\|\p_1\|_{\s,p}+\|\p_2\|_{\s,p})\|\p_1-\p_2\|_{1,p}.
\end{aligned}$\\
\\
Then, by combining the last two results, one obtains the assertion \eqref{phi2}.
\end{proof}

\begin{lemma}\label{lips}
Let us assume $[\mathcal{H}_2]$ then $ F  :\D ^{\s,p}(\T^n)\rightarrow L^p_0(\T^n)$ is locally Lipschitz continuous. 
\end{lemma}

\begin{proof}
It is sufficient to prove the local Lipschitz continuity of $\p\mapsto \nabla V . \nabla \p+\Delta V\p$, $\p\mapsto\prx (\pps )\p$ and $\p\mapsto\pps \prx \p$. First of all, the application $\D ^{\s,p}(\T^n) \ni \p \mapsto \prx (\pps )\p+\pps\prx \p \in L^p(\T^n)$ is well defined. Indeed, using the continuous embedding \eqref{boun} and Lemma~\ref{finft},

$\begin{aligned}
\|\prx \pps \p+\pps \prx \p\|_p&\leq \|\p\|_{p}\|\prx \pps \|_{\infty}+\|\pps \| _{\infty}\|\prx \p\|_p\notag\\
&\leq \frac{C}{\min \op}\|\p\|_{C^1}\|\p\|_{p}+C\|\p\|_{\s ,p}< \infty.
\end{aligned}$

Let $\p\in \D ^{\s,p}(\T^n)$ and $R>0$ such that $B_{\s,p}(\p,R)\subset\D ^{\s,p}(\T^n)$. Let $\p_1$ and $\p_2\in \mathcal{B}_{\s,p}(\p,R)$. We may assume without loss of generality that $R<\min \op $. With this choice of $R$, for $i=1,2$, $\op_i$ is bounded from below by a positive constant. Indeed, since $\p_i \geq \p-\|\p-\p_i\|_{\infty}$, then $\op _i\geq \op -\|\p-\p_i\|_{\infty}\geq\min \op -R=:\alpha$. Now, let us first consider
\begin{align}
\|\nabla V . \nabla \p_1+\Delta V\p_1 -\nabla V . \nabla \p_2-\Delta V\p_2 \|_p&=\|\nabla V . \nabla (\p_1-\p_2) +\Delta V(\p_1-\p_2)\|_p\notag\\
& \leq C\| \nabla (\p_1-\p_2)\|_p+C\|\p_1-\p_2\|_p\notag\\
& \leq C \| \p_1-\p_2\|_{\s,p}.\notag
\end{align}
 In addition, using the continuous embedding \eqref{boun}, Lemma \ref{finft} and Lemma~\ref{majfi}, one obtains
\begin{align}
\|\p_1\prx \phi_{\p_1} -\p_2\prx \phi_{\p_2}\|_p&\leq \|\prx \phi_{\p_1}(\p_1 -\p_2)\|_p+\|\p_2(\prx \phi_{\p_1} -\prx \phi_{\p_2})\|_p\notag\\
&\leq \|\prx \phi_{\p_1}\|_{\infty}\|\p_1 -\p_2\|_{p}+\|\p_2\|_{\infty}\|\prx (\phi_{\p_1} -\phi_{\p_2})\|_p\notag\\
&\leq \frac{C}{\min \op_1}\|\p_1\|_{\s,p}\|\p_1 -\p_2\|_{\s ,p}+C_{\p_1,\p_2}^2\|\p_2\|_{\s ,p}\|\p_1-\p_2 \|_{\s,p}\notag\\
&\leq \frac{C}{\alpha}(\|\p\|_{\s ,p}+R)\|\p_1 -\p_2\|_{\s ,p}\notag\\
&\quad + C\lp\frac{\|\p\|_{\s ,p}+R}{\alpha^2} +\frac{2(\|\p\|_{\s ,p}+R)^3}{\alpha^4}\rp\|\p_1 -\p_2\|_{\s ,p}\notag
\end{align}
and
\begin{align}
\|\phi_{\p_1}\prx \p_1 -\phi_{\p_2}\prx \p_2\|_p&\leq \|\phi_{\p_1}\prx (\p_1 -\p_2)\|_p+\|\prx \p_2(\phi_{\p_1} -\phi_{\p_2})\|_p\notag\\
&\leq \|\phi_{\p_1}\|_{\infty}\|\prx (\p_1 -\p_2)\|_{p}+\|\prx \p_2\|_{\infty}\|\phi_{\p_1}-\phi_{\p_2}\|_{p}\notag\\
&\leq C\|\p_1 -\p_2\|_{\s ,p}+C_{\p_1,\p_2}^1\|\p_2\|_{\s ,p}\|\p_1-\p_2\|_{p}\notag\\
&\leq C \lp 1+\frac{\|\p\|_{\s ,p}+R}{\alpha^2} \rp\|\p_1 -\p_2\|_{\s ,p}.\notag
\end{align}
This concludes the proof of Lemma~\ref{lips}.
\end{proof}

\begin{proposition}\label{loc0}
Let us assume $[\mathcal{H}_1]$ and $[\mathcal{H}_2]$. There exists $\delta>0$ and a unique $\p \in C([0,\delta],\mathcal{D}^{\sigma,p}(\mathbb{T}^n))$ such that,
\begin{equation}\label{mldn0}
\forall t\in [0, \delta], \quad\p  (t)-1=\mathrm{e}^{-t A_p }(\p _0-1)+\displaystyle \int_{0}^{t}\mathrm{e}^{-(t-s) A_p }F (\p  (s))ds.
\end{equation}
\end{proposition}
The local existence of a mild solution is a standard consequence of the fact that $ F  :\D ^{\s,p}(\T^n)\rightarrow L^p_0(\T^n)$ is locally Lipschitz continuous (ee \citep{am:84}, Proposition~2.1 or \ref{prloc0} for a proof).

\begin{remark}
Referring to Section~\ref{intro}, we have actually proved that there exists $\delta>0$ and a unique $\p \in C([0,\delta],\mathcal{D}^{\sigma,p}(\mathbb{T}^n))$ such that,
\begin{equation}\label{mldn}
\forall t\in [0, \delta], \quad \p (t)=\mathrm{e}^{\beta^{-1}t\Delta }\p_0+\displaystyle \int_{0}^{t}\mathrm{e}^{\beta^{-1}(t-s)\Delta } F  (\p (s))ds.
\end{equation}
For the sake of simplicity, we will use, in the following, this formulation instead of \eqref{mldn0}.
\end{remark}

The local existence result of a maximal mild solution is then obtained using standard arguments (see~\ref{prmldexs} for a proof).

\begin{theorem}\label{mldexs}
Assume $[\mathcal{H}_1]$ and $[\mathcal{H}_2]$, there exists a unique maximal mild solution $\p \in C([0,T_{\text{max}}),\mathcal{D}^{\sigma,p}(\mathbb{T}^n))$.
\end{theorem}

\subsection{Regularity of the solution}\label{reg}
In this section, we show that the maximum mild solution $\p$ built in the previous section is actually a $L^p-$solution of \eqref{fpp}. First we need to prove several preliminary lemmas.

The proof of the following lemma is rather standard, see \citep{am:84}, Proposition~1.4 and~\ref{prlipsh1}.
\begin{lemma}\label{lipsh1}
Assume $\p_0\in\D^{\s,p}(\T^n)$. Suppose that $\p :[0,T_{\text{max}})\rightarrow W^{\s,p}(\mathbb{T}^n)$ is the maximal mild solution of \eqref{fpp}, then\begin{equation}\label{nu}
\displaystyle t\mapsto v(t):=\int_0^t\e^{-(t-s) A_p } F  (\p(s))ds\in C_{\text{loc}}^{\lp1-\frac{\s}{2}\rp^-}([0,T_{\text{max}}),\D^{\s,p}(\T^n)).
\end{equation}
\end{lemma}

\begin{lemma}\label{lipsh2}
For $\p_0\in\D^{\s,p}(\T^n)$ and $0\leq \tau<\s\leq 2$, then $\forall T>0$
\begin{equation}\label{to}
t\mapsto \mathrm{e}^{\beta^{-1}t \Delta}\p _0\in C^{\frac{\sigma-\tau}{2}}([0,T],W^{\tau,p}(\mathbb{T}^n)).
\end{equation}
\end{lemma}
\begin{proof}
For $t=0$ and $h>0$, then by \eqref{reg2}\\

$\begin{aligned}
\left\|(\e^{-h A_p }-I)(\p_0-1)\right\|_{\tau,p}&=\left\| A_p ^{\tau/2}(\e^{-h A_p }-I)(\p_0-1)\right\|_p\\
&=\left\|(\e^{-h A_p }-I) A_p ^{\tau/2}(\p_0-1)\right\|_p\\
&\leq C_{\tau,\sigma}h^{\frac{\s-\tau}{2}}\left\| A_p ^{\frac{\sigma-\tau}{2}} A_p ^{\tau/2}(\p_0-1)\right\|_p\\
&=C_{\tau,\sigma}h^{\frac{\sigma-\tau}{2}}\|(\p_0-1)\|_{\s,p}.
\end{aligned}$\\

Now for $t>0$, using the previous case and the fact that $\e^{-t A_p }$ is a contraction semigroup, one gets\\

$\begin{aligned}
\left\|(\e^{-(t+h) A_p }-\e^{-t A_p })(\p_0-1)\right\|_{\tau,p}&=\left\|(\e^{-h A_p }-I)\e^{-t A_p }(\p_0-1)\right\|_{\tau,p}\\
&\leq C_{\tau,\sigma}h^{\frac{\sigma-\tau}{2}}\|\e^{-t A_p }(\p_0-1)\|_{\s,p}\\
&=C_{\tau,\sigma}h^{\frac{\sigma-\tau}{2}}\|\e^{-t A_p } A_p ^{\frac{\s}{2}}(\p_0-1)\|_{p}\\
&\leq C_{\tau,\sigma}h^{\frac{\sigma-\tau}{2}}\| A_p ^{\frac{\s}{2}}(\p_0-1)\|_{p}\\
&\leq C_{\tau,\sigma}h^{\frac{\sigma-\tau}{2}}\|(\p_0-1)\|_{\s,p}.
\end{aligned}$\\

\noindent Then $\mathrm{e}^{-t A_p }(\p _0-1)$ (and therefore $\mathrm{e}^{\beta^{-1}t \Delta}\p _0$) belongs to $ C^{\frac{\sigma-\tau}{2}}_{\text{loc}}([0,T_{\text{max}}),W^{\tau,p}(\mathbb{T}^n))$.\\

\end{proof}

\begin{lemma}\label{holfp}
Let us assume $[\mathcal{H}_1]$ and $[\mathcal{H}_2]$. Suppose that $\p :[0,T_{\text{max}})\rightarrow \D^{\s,p}(\mathbb{T}^n)$ is the maximal mild solution of \eqref{fpp}. Then 
\begin{enumerate}
\item  $\forall \varepsilon \in [0,1-\frac{\s}{2}) $, $\forall \tau \in [0,\s)$,\begin{equation}\label{plip}
\p\in C^\mu_{\text{loc}}([0,T_{\text{max}}),\D^{\tau,p}(\T^n)),\text{ where }\mu:=\min\lp 1-\frac{\s}{2}-\varepsilon,\frac{\s-\tau}{2} \rp.
\end{equation}

\item  \begin{equation}\label{flip}
F  (\p)\in C^{\nu^-}_{\text{loc}}([0,T_{\text{max}}),L^p_0(\T^n)),\text{ where }\nu:=\min\lp 1-\frac{\s}{2},\frac{\s}{2}-\frac{1}{2}-\frac{n}{2p} \rp.
\end{equation}
\end{enumerate}
\end{lemma}

\begin{proof}
For the first assertion, using the embedding $W^{\sigma,p}\hookrightarrow W^{\tau,p}$ (since $\tau<\s$), \eqref{nu} and \eqref{to}, one obtains that
$$\p \in C^{\mu}([0,T_{\text{max}}),\D^{\tau,p}(\mathbb{T}^n)),\mbox{ with } \mu:=\min\lp 1-\frac{\s}{2}-\varepsilon ,\frac{\sigma-\tau}{2}\rp .$$
For the second assertion, let $\tau\in (1+\frac{n}{p},\s) $, using again the embedding $W^{\sigma,p}\hookrightarrow W^{\tau,p}$  and applying Lemma~\ref{lips} with $\s$ replaced by $\tau\in (1+\frac{n}{p},\s)$, one gets
\begin{equation}\label{clip}
F \in  C^{1^-}_{loc}(\mathcal{D}^{\tau,p}(\T^n),L^p_0(\T^n)).
\end{equation}
The results \eqref{plip} and \eqref{clip} imply \eqref{flip}. Indeed, let $0<t\leq T<T_{\text{max}}$. For all $ t\in [0,T]$, there exists a positive real number $ \alpha(t)$ such that $F$ is Lipschitz on $B_{\tau,p}(\p(t),\alpha(t))$ since  $F \in  C^{1^-}_{loc}(\mathcal{D}^{\tau,p}(\T^n),L^p(\T^n))$. We know that $\p([0,T])=\{\p(t),\,t\in [0,T]\}$ is a compact set of $\D^{\tau,p}(\T^n)$. Then, $\exists t_1,...,t_n \in [0,T] $ and $\alpha(t_i)>0$, $i\in \{1,...,n\}$ such that $ F  $ is Lipschitz on $B_{\tau,p}(\p(t_i),\alpha(t_i))$, $\forall i\in \{1,...,n\}$ and
$$\p([0,T])= \displaystyle \bigcup_{i=1}^n\lc B_{\tau,p}\lp\p(t_i),\frac{\alpha(t_i)}{2}\rp\bigcap \p([0,T])\rc .$$
Define now $\alpha:=\displaystyle\min_{1\leq i\leq n}\frac{\alpha(t_i)}{2}$. Since $\p\in C([0,T],\D^{\tau,p}(\T^n))$, then it is uniformly continuous on $[0,T]$:
$$\exists\varepsilon>0,\text{ such that }\forall s,t\in [0,T],\,|t-s|\leq \varepsilon,\,\|\p(t)-\p(s)\|_{\tau,p}\leq \alpha. $$
Let $t,s\in [0,T]$ such that $\|\p(t)-\p(s)\|_{\tau,p}\leq \alpha$. Then $\exists i\in \{1,...,n\}$ such that $ \|\p(t)-\p(t_i)\|_{\tau,p}\leq \frac{\alpha(t_i)}{2} $. Consequently, \\
$\begin{aligned}
\|\p(s)-\p(t_i)\|_{\tau,p}&\leq \|\p(s)-\p(t)\|_{\tau,p}+\|\p(t)-\p(t_i)\|_{\tau,p}\\
&\leq \alpha +\frac{\alpha(t_i)}{2}\leq \alpha(t_i).
\end{aligned}$\\

Then, for $t,s\in [0,T]$, if $|t-s|\leq \varepsilon$, $\exists i\in \{1,...,n\}$ such that $\p(t)$ and $\p(s)$ belong to $B_{\tau,p}(\p(t_i),\alpha(t_i))$ and using the fact that $F$ is Lipschitz on $B_{\tau,p}(\p(t_i),\alpha(t_i))$ then, using \eqref{plip}
$$\| F  (\p (t))- F  (\p (s))\|_p\leq C_0\|\p(t)-\p(s)\|_{\tau,p}\leq C_1|t-s|^{\mu},$$
where $C_0$ is the Lipschitz constant on $\displaystyle\bigcup_{i=1}^n B_{\tau,p}(\p(t_i),\alpha(t_i))$. If $|t-s|>\varepsilon $, then since $ F  \in C(\D^{\tau,p}(\T^n),L^p_0(\T^n))$, then $\displaystyle C_2:=\sup_{t\in[0,T]}\| F  (\p(t))\|_p<\infty$ and
$$\| F  (\p (t))- F  (\p (s))\|_p\leq 2C_2\leq \frac{2C_2}{\varepsilon^\mu}|t-s|^\mu.$$
In conclusion, we have that
$$\forall t,s\in[0,T],\| F  (\p (t))- F  (\p (s))\|_p\leq C|t-s|^\mu,$$
where $C=\max(C_1,\frac{2C_2}{\varepsilon^\mu})$. Observe that $\displaystyle \nu:=\varlimsup_{\underset{\tau\rightarrow 1+\frac{n}{p}}{\varepsilon\rightarrow 0}}\mu$, thus $\forall \tilde{\varepsilon}>0 $, $\exists \varepsilon>0 $ and $\exists \tau>1+\frac{n}{p}$, such that $\mu\geq \nu- \tilde{\varepsilon}$. 
\end{proof}
The proofs of the two following results are similar to the proofs of Lemma~3.4 and Theorem~3.5 in \citep{paz:83}. We provide details of the proof in \ref{prv1lp} and \ref{prmilp}.

\begin{lemma}\label{v1lp}
Assume $[\mathcal{H}_2]$ and suppose that $\p :[0,T_{\text{max}})\rightarrow W^{\s,p}(\mathbb{T}^n)$ is the maximal mild solution of \eqref{fpp}. For $t\in [0,T_{\text{max}})$, define
\begin{equation}\label{v1}
v_1(t):=\displaystyle\int_0^t\e^{-(t-s) A_p }\lp  F  (\p(s))- F  (\p(t))\rp ds.
\end{equation}
Then $\forall t\in [0,T_{\text{max}})$ $v_1(t)\in D( A_p )$ and $ A_p v_1\in C^{\nu^-}_{\text{loc}}([0,T_{\text{max}}),L^p(\T^n))$, where $\nu$ is defined in \eqref{flip}.
\end{lemma}

\begin{theorem}\label{milp}
Assume $[\mathcal{H}_1]$ and $[\mathcal{H}_2]$. Suppose that $\p :[0,T_{\text{max}})\rightarrow W^{\s,p}(\mathbb{T}^n)$ is the maximal mild solution of \eqref{fpp}, then 
\begin{itemize}
\item [$(i)$] $\Delta\p\in C^{\nu^-}_{\text{loc}}((0,T_{\text{max}}),L^p(\T^n))$ and $\dot{\p}\in C^{\nu^-}_{\text{loc}}((0,T_{\text{max}}),L^p(\T^n))$, where $\nu^-$ is defined in \eqref{flip};\\
\item [$(ii)$] If in addition $\p_0\in W^{2,p}(\T^n)$ then $\Delta \p$ and  $\dot{\p}$ are continuous on $[0,T_{\text{max}})$ with values in $L^p(\T^n)$. 
\end{itemize}
\end{theorem}

\subsection{Proof of Theorem~\ref{loc}}\label{prth}

We are now in the position to prove Theorem~\ref{loc}. The local existence of the maximal $L^p-$solution is a consequence of Theorem \ref{mldexs} and Theorem \ref{milp}$(i)$. In the following we prove the first item $(i)$ of Theorem~\ref{loc}.

\medskip

For all $t\in(0,T_{\text{max}})$, we have $\p(t)-1\in W^{2,p}(\T^n)\cap L^p_0(\T^n)$ and $\frac{d}{dt}(\p-1)(t) \in L^p_0(\T^n)$, then we have in $L^p_0$
$$ A_p (\p(t)-1)= F  (\p(t))-\frac{d}{dt}(\p(t)-1).$$
Since $ A_p ^{-1}\in \mathcal{L}(L^p_0(\T^n),W^{2,p}(\T^n)\cap L^p_0(\T^n))$, $ F  (\p(\cdot))\in C([0,T_{\text{max}}),L^p_0(\T^n))$ and $\frac{d}{dt}(\p(t)-1)\in C((0,T_{\text{max}}),L^p_0(\T^n))$ then
$$\p(t)-1= A_p ^{-1} F  (\p(t))- A_p ^{-1}\frac{d}{dt}(\p(t)-1)$$ 
is continuous from $(0,T_{\text{max}})$ into $D( A_p )=W^{2,p}(\T^n)\cap L^p_0(\T^n)$. The rest of $(i)$ follows from Lemma~\ref{holfp}. 
\medskip

The part $(ii)$ of Theorem~\ref{loc} follows from Theorem~\ref{milp}$(i)$. We are now in position to prove $(iii)$. 

\medskip

It follows from Lemma~\ref{holfp} that, $\forall \varepsilon  \in [0,1-\s/2) $, $\forall \tau \in [0,\s)$, $\p\in C^\mu([0,T],$ $W^{\tau,p}(\T^n))$, where $\mu:=\min\lp 1-\frac{\s}{2}-\varepsilon,\frac{\s-\tau}{2} \rp $. Let $\tilde{\varepsilon}=\min\lp\frac{1}{3}(\s-1-\frac{n}{p}),\frac{2-\s}{2} \rp>0 $ and let $\tau=1+\frac{n}{p}+\rho$, for all $0<\rho<\tilde{\varepsilon}$. Now, by \eqref{boun}, we have $W^{\tau,p}(\T^n)\hookrightarrow C^{1+\rho^-}(\T^n)$. To prove $(iii)$, it is sufficient to prove that $\mu\geq\rho $, which holds since:
 \begin{enumerate}
 \item  $\frac{\s-\tau}{2}>\rho $. This is equivalent to $\s-1-\frac{n}{p}>3\rho$, which is true since $\rho<\tilde{\varepsilon}<\frac{1}{3}(\s-1-\frac{n}{p})$.
 \item  $1-\frac{\s}{2}>\rho$, which is true by the definition of $\tilde{\varepsilon}$.
 \end{enumerate}
  Thus $\p\in C^\rho([0,T],C^{1+\rho^-}(\T^n))$. Now, we look for the largest value of $\rho$ and thus of $\tilde{\varepsilon}$. Let us first optimize $\s\in(1+\frac{n}{p},2)$. In view of the definition of $\tilde{\varepsilon}$, the best $\s$ (denoted $\s_{\text{opt}}$ in the following), is such that
 $$\displaystyle \frac{1}{3}\lp\s_{\text{opt}}-1-\frac{n}{p}\rp =\frac{2-\s_{\text{opt}}}{2} \Longleftrightarrow \s_{\text{opt}}=\frac{1}{5}\lp8+\frac{2n}{p}\rp.$$
It is easy to check that $\s_{\text{opt}}\in(1+\frac{n}{p},2)$. Therefore, the optimized value of $\tilde{\varepsilon}$ (denoted $\tilde{\varepsilon}_{\text{opt}}$ in the following), satisfies:
$$\tilde{\varepsilon}_{\text{opt}}=1-\frac{ \s_{\text{opt}}}{2}=\frac{1}{5}\lp 1-\frac{n}{p}\rp. $$
Therefore, $\p\in C^{\tilde{\varepsilon}_{\text{opt}}}([0,T],C^{1+\tilde{\varepsilon}_{\text{opt}}}(\T^n))$. Finally, when $p\rightarrow +\infty $, $\tilde{\varepsilon}_{\text{opt}}\rightarrow\frac{1}{5}$, which implies that if $\p_0\in\D^{\frac{8}{5},\infty}(\T^n)$, then $\p\in C^\rho([0,T],C^{1+\rho}(\T^n))$ for all $\rho\in (0,\frac{1}{5})$.  
\medskip

It remains to prove $(iv)$. To get the $C^1((0,T_{\text{max}}),C^2(\T^n))$ regularity, we consider the parabolic problem as a linear problem with H\"older-continuous right-hand side. Indeed, we have that 
$$\dot{\p} -\beta^{-1}\Delta \p= F  (\p)\qquad \mbox{in}\,(0,T_{\text{max}}).$$
Fix $0<\delta<T_{\text{max}}$ and define the following cut-off function $\kappa\in C^{\infty}(\R )$ such that
\begin{equation*}
\kappa(t)=\left\{
\begin{aligned}
0&\mbox{ for }\,t\leq \delta/2,\\
1&\mbox{ for }\,t> \delta.
\end{aligned}
\right.
\end{equation*}
Let $v(t):=\kappa(t)\p(t)$. Thus, one obtains the following linear parabolic problem 
\begin{equation}\label{lin}
\left\{
\begin{aligned}
\dot{v} -\beta^{-1}\Delta v&=\tilde{f} \mbox{ in }(0,T_{\text{max}}),\\
v(0)&=0,
\end{aligned}\right.
\end{equation}
where 
\begin{equation}\label{tf}
\tilde{f}(t):= \kappa(t) F  (\p(t))+\dot{\kappa}(t)\p(t).
\end{equation}
 In the following, we prove that there exists $\alpha\in (0,1)$ such that $\tilde{f}\in C^{\alpha/2}((0,T_{\text{max}}),$ $ C^{\alpha}(\T^n))$. The last assertion is satisfied as soon as we prove that $ \tilde{f}\in C^{\gamma}((0,T_{\text{max}}),$ $C^{1+\gamma}(\T^n))$, for some $\gamma\in (0,1)$ by taking $\alpha:=2\gamma$ (since $1+\gamma>2\gamma=\alpha$).

\medskip

We know that $\p\in C^{\rho^-}([0,T_{\text{max}}),C^{1+\rho^-}(\T^n))$. Now, by a bootstrap argument, we prove $F(\p)\in C^{\rho^-}([0,T_{\text{max}}),C^{1+\rho^-}(\T^n))$. Recall that 
\begin{equation}\label{a0}
\p\in C((0,T_{\text{max}}),W^{2,p}(\T^n))\cap C^1((0,T_{\text{max}}),L^p(\T^n)).
\end{equation}
We know that $ F  (\p)\in C([0,T_{\text{max}}),L^p(\T^n))$ and
$$ F  (\p)=\Delta V\p+\nabla V\cdot\nabla\p-\frac{\int \pr_1 V\p}{\int \p}\pr_1\p-\frac{\int \pr_{11}^2 V\p}{\int \p}\p-\frac{\int \pr_1 V\pr_1\p}{\int \p}\p+\frac{\int \pr_1 V\p}{\int \p}\frac{\int \pr_1 \p}{\int \p}\p. $$
Then by \eqref{a0}, it is easy to show that
\begin{equation}\label{b0}
 F  (\p)\in  C((0,T_{\text{max}}),W^{1,p}(\T^n)).
\end{equation}
Since $\dot{\p}-\beta^{-1}\Delta \p= F  (\p)$, then by standard $L^p$ regularity for the heat kernel
\begin{equation}\label{a1}
\p\in C((0,T_{\text{max}}),W^{3,p}(\T^n))\cap C^1((0,T_{\text{max}}),W^{1,p}(\T^n))).
\end{equation}
Differentiating $ F  (\p)$ in space, one obtains:
$$\pr_{i} F  (\p)=\pr_{i}(\nabla V . \nabla \p) + \pr_{i}(\Delta V \p) -\pr_{i1}\pps\p-\pr_{1}\pps\pr_{i}\p-\pr_{i}\pps\pr_{1}\p-\pps\pr_{i1}\p.$$
When $i\neq 1$, $\pr_{i}\pps=\pr_{i1}\pps=0$. When $i=1$,
$$\begin{aligned}\label{pr11}
\pr_{11}\pps&=\frac{\int \pr_{111}^3V\p}{\int \p}+2\frac{\int \pr_{11}^2V\pr_{1}\p}{\int \p}-2\frac{\int \pr_{11}^2V\p}{\int \p}\frac{\int \pr_{1}\p}{\int \p}\notag\\
&\quad +\frac{\int \pr_1 V\pr_{11}^2\p}{\int \p}-2\frac{\int \pr_1 V\pr_{1}\p}{\int \p}\frac{\int \pr_{1}\p}{\int \p}+2\frac{\int \pr_1 V\p}{\int \p}\lp \frac{\int \pr_{1}\p}{\int \p}\right )^2\notag\\
&\quad -\frac{\int \pr_1 V\p}{\int \p}\frac{\int \pr_{11}^2\p}{\int \p}.
\end{aligned}$$
Since there exists $\alpha >0$ such that $\frac{1}{\min \op}<\frac{1}{\alpha}$, then one can easily prove
\begin{equation}\label{b1}
 F  (\p)\in C ((0,T_{\text{max}}),W^{2,p}(\T^n)).
\end{equation}
Similarly since $\dot{\p}-\beta^{-1}\Delta \p= F  (\p)$, we have:
\begin{equation}\label{a2}
\p\in C((0,T_{\text{max}}),W^{4,p}(\T^n))\cap C^1((0,T_{\text{max}}),W^{2,p}(\T^n))).
\end{equation}
By iterating the arguments one more time, the following regularity result is satisfied
\begin{equation}\label{a3}
\p\in C((0,T_{\text{max}}),W^{5,p}(\T^n))\cap C^1((0,T_{\text{max}}),W^{3,p}(\T^n))).
\end{equation}
Differentiating $ F(\p)$ in time, one then obtains:
 $$\pr_{t} F  (\p)=\nabla V . \nabla (\pr_{t}\p) + \Delta V (\pr_{t}\p) -\pr_{t}\pr_{1}\pps\p-\pr_{1}\pps\pr_{t}\p-\pr_{t}\pps\pr_{1}\p-\pps\pr_{1}\pr_{t}\p.$$
And it follows without difficulties that
\begin{equation}\label{b2}
 F  (\p)\in C^1 ((0,T_{\text{max}}),W^{2,p}(\T^n)).
\end{equation}
Then $\exists \rho' \in (0,1)$ such that we have the embedding $W^{2,p}(\T^n)\hookrightarrow C^{1+\rho'}(\T^n) $. Therefore, $\exists \rho '\in (0,1)$ such that
\begin{equation}
 F  (\p)\in C^{\rho'}((0,T_{\text{max}}),C^{1+\rho'}(\T^n)).
\end{equation}
Taking $\gamma=\min(\rho,\rho') $, one obtains that $ \tilde{f}\in C^{\gamma}((0,T_{\text{max}}),C^{1+\gamma}(\T^n))$, defined by \eqref{tf}. Finally, one can now use Theorem 48.2(ii) in \citep{qui:07}, to show that there exists a unique classical solution $w$ of \eqref{lin} satisfying $w \in C^{1}( (0,T_{\text{max}}),C^2(\T^n))$. By uniqueness of solutions of \eqref{lin} we have that $w=v$, therefore $\p$ is a classical solution for $t>0$.$\square$ 

\begin{remark}(Another method to deal with the nonlocal term)\label{secd}\\
We present a second way to handle the nonlocal term in the proof of local-in-time existence and uniqueness result. The initial problem \eqref{fpp} can be written as
\begin{equation}\label{scdmtd}
\pr_t(\p-1)-A_p(\p-1)=H(t,\p),
\end{equation}where 
$$H(t,\p):={\rm div}(\nabla V \p)-\pr_1\lp\frac{\displaystyle\int_{\T^{n-1}} \pr_1 V\p dx_2...dx_n}{\tp(t,x_1) } \rp$$
with $\tp $ satisfying the following diffusion equation:\\
\begin{equation}
\left\{
\begin{aligned}
\pr_t\tp &=\pr_{x_1x_1}^2\tp  \mbox{ in }\T\times [0,\infty) ,\\
\tp (0,.)&=\displaystyle\int_{\T^{n-1}} \p(0,x)dx_2...dx_n  \mbox{ on }\, \T.\\
\end{aligned}
\right.
\end{equation}
Notice that,  if $\op _0>0$, then for all $t>0$ , $\tp (t)>0$. One can prove that $H:(0,T_{\text{max}})\times W^{\s,p}(\T^n) \rightarrow L^p(\T^n)$ is locally Lipschitz continuous, which implies the local existence of a solution to \eqref{scdmtd}. It is then easy to show that $\tp (t,x_1)=\displaystyle\int_{\T^{n-1}}\p (t,x) dx_2...dx_n$. Indeed, by integrating the equation \eqref{scdmtd} which writes:
$$\pr_t \p=\beta^{-1}\pr_{x_1x_1} \p+{\rm div} (\nabla V\cdot \p)-\pr_1\lp\frac{\displaystyle\int_{\T^{n-1}} \pr_1 V\p }{\tp } \p\rp,$$
one obtains,
$$\pr_t\op=\beta^{-1}\pr_{x_1x_1}\op+\prx \lp\int_{\T^{n-1}} \pr_1 V \p \rp-\prx \lp\frac{\int_{\T^{n-1}} \pr_1 V \p }{\tp }\op \rp,$$
where $\op(t,x_1)=\displaystyle\int \p dx_2...dx_n $. Denote $f(t,x_1)=\int_{\T^{n-1}} \pr_1 V \p dx_2...dx_n$, we obtain the equation:
\begin{equation}\label{hf}
\pr_t \op(x_1)=\beta^{-1}\pr_{x_1x_1}\op(x_1)+\prx f(t,x_1)-\prx \lp\frac{f(t,x_1)}{\tp }\op\rp.
\end{equation}
It is easy to see that $\tp $ is a solution of \eqref{hf} and $\tp(0,x_1)=\op(0,x_1)$. Since \eqref{hf} admits a unique solution, then $\tp =\op$.
\end{remark}

\section{ Diffusion equation and weak maximum principle}\label{max}
In this Section, we first prove that $\op $ satisfies a simple diffusion equation and we show that the solution of \eqref{fpp} is positive.

\noindent\textbf{Proof of Proposition~\ref{diff}.}\\
 Let $g:\T \rightarrow\R $ be a function in $H^1(\T),$ and let $\p$ is a $L^p-$solution of \eqref{fpp}, then we have:
 $$\begin{aligned}
\frac{d}{dt}\displaystyle\int_{\T}\op (t,x_1)g(x_1)dx_1&=\frac{d}{dt}\displaystyle\int_{\T^n}\p(x)g(x_1)dx\\
&=\displaystyle\int_{\T^n}\big[{\rm div}(\nabla V\p+\beta^{-1}\nabla\p)g(x_1)-\prx (\pps \p)g(x_1)\big ] dx\\
&=\displaystyle\int_{\T^n}\big[-(\prx  V\p+\beta^{-1}\prx \p)g'(x_1)+\pps \p g'(x_1)\big ] dx\\
&=-\beta^{-1}\displaystyle\int_{\T^n}\prx \p(x)g'(x_1)dx\\
&=-\beta^{-1}\displaystyle\int_{\T}\prx \op (t,x_1)g'(x_1)dx_1\\
&=\beta^{-1}\displaystyle\int_{\T}\pr_{x_1x_1}^2\op (x)g(x_1)dx_1,
\end{aligned}$$
 which is a weak formulation in the distribution sense of :
$$\pr_t\op =\beta^{-1}\pr_{x_1x_1}^2\op ,\, \text{in}\,[0,\infty)\times \T.$$
Finally, using the fact that $\p$ is a density function, then $\displaystyle\int_\T\op =\displaystyle\int_{\T^n}\p=1$. $\square$

Let $\p$ be the maximum mild solution of \eqref{fp2}. We are now in position to prove the positivity of $\p$.\\

\noindent \textbf{Proof of Proposition~\ref{pos1}.}\\
Multiplying $\dot{\p}(t)=\beta^{-1}\Delta \p(t)+{\rm div}(\nabla V\p-\pps e_1\p)$ (where $e_1$ is defined in Section~\ref{not}) by the negative part of $\p$ defined as
$$\p_{-}=\frac{|\p|-\p}{2}\geq 0,$$
and integrating over $\T^n$, one gets
\begin{equation}\label{maxpr}
\displaystyle \int_{\T^n} \dot{\p}\p_{-}=\int_{\T^n} \beta^{-1}\Delta \p\p_{-}+\int_{\T^n}{\rm div}(\nabla V\p-\pps e_1\p)\p_{-}.
\end{equation}
By the definition of $\p_{-}$ and using properties about $\nabla \p_{-}$ stated in Lemma $7.6$ in \citep{gil:77}, one obtains:
$$\displaystyle \int_{\T^n} \dot{\p}\p_{-}=-\frac{1}{2}\frac{d}{dt}\int_{\T^n} \p_{-}^2$$
 and
$$\displaystyle \int_{\T^n} \beta^{-1}\Delta \p\p_{-}=\beta^{-1} \int_{\T^n} |\nabla \p_{-}|^2.$$
We now restrict ourselves to $[0, T]$, for a fixed $T<T_{\text{max}}$. Referring to Lemma~\ref{finft} and applying Young's inequality,\\

$\begin{aligned}
\left|\int_{\T^n}{\rm div}(\nabla V\p-\pps e_1\p)\p_{-}\right|&=\left|\int_{\T^n}(\nabla V\p-\pps e_1\p)\cdot\nabla\p_{-}\right|\\
&\leq 2\|\nabla V\|_\infty^2\varepsilon\int_{\T^n} |\nabla \p_{-}|^2+C_{\varepsilon}\int_{\T^n}|\p_{-}|^2,
\end{aligned}$\\

where $\varepsilon>0$ is arbitrary, but $C_{\varepsilon}$ depends on the choice of $\varepsilon$. Choosing $\varepsilon>0$  so that $-\beta^{-1}+2\|\nabla V\|_\infty^2\varepsilon \leq 0$, then \eqref{maxpr} becomes
$$\displaystyle \frac{1}{2}\frac{d}{dt}\int_{\T^n} |\p_{-}|^2 \leq C_{\varepsilon}\int_{\T^n}|\p_{-}|^2.$$
Having $\p \in C([0,T_1],L^2(\T^n))$ (since $2<p$) and $\p(0)_-=0$ (since $\p_0$ is supposed positive), Gronwall's lemma now implies $\int_{\T^n} |\p_{-}|^2=0$, for all $t\in [0, T]$. Since $T<T_{\text{max}}$ is arbitrary, we see that $\p(t)^{-}= 0$, for all $t\in (0, T_{\text{max}})$. 

\section{A-priori estimates for solutions and global existence}\label{priob}
In this Section, we prove some a-priori bounds for $F$ and universal a-priori bounds for $\p$, which are essential to prove global existence. We will use repeatedly the fact that $\op\geq \alpha >0 $ on $(0,T_{\text{max}})$, since $\op_0\in\D^{\s,p}(\T^n) $ and $\op$ satisfies $\pr_t\op =\beta^{-1}\pr_{x_1x_1}\op$.

\subsection{Polynomial and universal a-priori bounds}\label{pub}

We define $\MM _0(x_1)=\displaystyle\sup_{x_2,...,x_n}(\p_0\e^{\frac{\beta V}{2}})(x_1,x_2,...,x_n)$ and we consider the following parabolic problem on a function $\MM :\R_+\times \T\rightarrow \R$,
\begin{equation}\label{mps}
\left\{
\begin{aligned}
\pr_t\MM  &=\beta^{-1}\partial_{x_1x_1}\MM -\delta\MM -\prx(\pps \MM )\text{ on }[0,+\infty)\times\T,\\
\MM (0,x_1)&=\MM _0(x_1),
\end{aligned}
\right.
\end{equation}
where 
$$\pps (t,x_1)=\frac{\displaystyle \int_{\T^{n-1}} \prx  V(x)\p (t,x) dx_2...dx_n}{\op (t,x_1)},$$
$$\op (t,x_1)= \displaystyle \int_{\T^{n-1}}\p (t,x) dx_2...dx_n,$$
$$\displaystyle \delta=-\max_{x\in\T^n} \lp\frac{\Delta V}{2}-\frac{\beta}{4}|\nabla V|^2+\frac{\beta|\prx V|}{2}\|\prx V\|_\infty\rp$$
and $\p$ is the unique maximal solution defined in Section~\ref{theo}.
\medskip

In a classical way, one can prove that the problem \eqref{mps} admits a unique solution $\MM \in L^{\infty}((0,T_{\text{max}}),L^2(\T))\cap L^{2}((0,T_{\text{max}}),H^1(\T))$ and satisfies $\MM \geq 0$ since $\pps\in L^\infty((0,T_{\text{max}}),L^\infty(\T^n))$ , $V\in C^2(\T^n)$ and $\MM _0\geq 0$. This function will be used to dominate the solution $\p$ given by Theorem~\ref{loc}. In fact, we have the following lemma.
\begin{lemma}\label{mp}
Let $\MM $ be the solution of \eqref{mps}, then the solution $\p$ given by Theorem~\ref{loc} satisfies:
$$\p\leq\MM \,\e^{\frac{-\beta V}{2}}.$$
\end{lemma}
\begin{proof}
The analysis will be carried out in a suitable system of coordinates which simplifies the calculations. We will perform two changes of variable. First, one may assume that $\beta=1$ up to the following change of variable: $\widehat{t}=\beta^{-1}t$, $\widehat{\p }(\widehat{t},x)=\p (t,x),$ $\widehat{V}(x)=\beta V(x)$. Second, we take:
$$\tp =\p\e^{\frac{V}{2}}. $$
Therefore, the problem \eqref{fpp} becomes
$$\begin{aligned}
\pr_t \tp&={\rm div}\lc \nabla V\tp\e^{-\frac{V}{2}} + \nabla(\tp \e^{-\frac{V}{2}}) \rc \e^{\frac{V}{2}}-\prx \lp \pps \tp \e^{-\frac{V}{2}}\rp \e^{\frac{V}{2}}\notag\\
&={\rm div}\lc \lp \frac{\nabla V}{2}\tp + \nabla\tp \rp \e^{-\frac{V}{2}}\rc \e^{\frac{V}{2}}-\prx \lp \pps \tp \rp +\frac{\prx V}{2}\pps\tp \notag\\
&= \frac{\Delta V}{2}\tp +\Delta\tp -\frac{\nabla V}{2}\cdot\frac{\nabla V}{2}\tp -\prx \lp \pps \tp \rp +\frac{\prx V}{2}\pps\tp \notag\\
&= \Delta\tp+W\tp-\prx \lp \pps \tp \rp +\frac{\prx V}{2}\pps\tp, \label{fpchg}
\end{aligned}$$
where $W=\lp \frac{\Delta V}{2}-\frac{|\nabla V|^2}{4}\rp $. Multiplying Equation \eqref{fpchg} by $(\tp-\MM )_+$ and integrating over space, one then obtains
$$\begin{aligned}
\displaystyle\int_{\T^n}\pr_t\tp(\tp-\MM )_+&=\int_{\T^n}\Delta\tp(\tp-\MM )_++\int_{\T^n}W\tp(\tp-\MM )_+\notag\\
&\quad -\int_{\T^n}\prx \lp \pps \tp \rp (\tp-\MM )_++\int_{\T^n}\frac{\prx V}{2}\pps\tp(\tp-\MM )_+.\label{fpchg2}
\end{aligned}$$

Furthermore,
$$\displaystyle\int_{\T^n}\pr_t\tp(\tp-\MM )_+=\frac{1}{2}\frac{d}{dt}\int_{\T^n}\lp(\tp-\MM )_+\rp^2+\int_{\T^n}\pr_t\MM (\tp-\MM )_+.$$
Integrating by parts, one thus obtains
$$\begin{aligned}
\displaystyle\frac{1}{2}\frac{d}{dt}\int_{\T^n}\lp(\tp-\MM )_+\rp^2&=-\int_{\T^n}\pr_t\MM (\tp-\MM )_+\\
&\quad-\int_{\T^n}\mid\nabla(\tp-\MM )_+ \mid ^2+\int_{\T^n}\partial_{x_1x_1}\MM (\tp-\MM )_+\\
&\quad +\int_{\T^n}W\lp(\tp-\MM )_+\rp^2 +\int_{\T^n}W\MM (\tp-\MM )_+\\
&\quad+\int_{\T^n} \pps (\tp-\MM ) \prx(\tp-\MM )_+-\int_{\T^n}\prx \lp \pps \MM  \rp (\tp-\MM )_+\\
&\quad +\int_{\T^n}\frac{\prx V}{2}\pps \lp(\tp-\MM )_+\rp^2+\int_{\T^n}\frac{\prx V}{2}\pps \MM (\tp-\MM )_+.
\end{aligned}$$
In addition, applying Young's inequality
$$\displaystyle\int_{\T^n} \pps (\tp-\MM ) \prx(\tp-\MM )_+\leq \frac{1}{4\varepsilon} \|\prx V\|_{\infty}^2 \int_{\T^n} \lp(\tp-\MM )_+\rp^2+\varepsilon \int_{\T^n}|\nabla (\tp-M)_+|^2.$$
Then, by Lemma~\ref{finft}
$$\begin{aligned}
&\displaystyle\frac{1}{2}\frac{d}{dt}\int_{\T^n}\lp(\tp-\MM )_+\rp^2\\
&\leq-(1-\varepsilon)\int_{\T^n}\mid\nabla(\tp-\MM )_+ \mid ^2+\lp\|W\|_{\infty}+\frac{1}{4\varepsilon} \|\prx V\|_{\infty}^2 \rp\int_{\T^n}\lp(\tp-\MM )_+\rp^2\\
&\quad+\int_{\T^n}\frac{\prx V}{2}\pps \lp(\tp-\MM )_+\rp^2\\
&\quad + \int_{\T^n}\lc  -\pr_t\MM + \partial_{x_1x_1}\MM +W \MM - \prx(\pps \MM )+\frac{\prx V}{2}\pps \MM  \rc (\tp-\MM )_+.
\end{aligned}$$
One knows that
$$-\pr_t\MM + \partial_{x_1x_1}\MM +W \MM - \prx(\pps \MM )+\frac{\prx V}{2}\pps \MM  \leq 0.$$
Indeed, since $\MM $ is the solution of \eqref{mps} (and since $\MM\geq 0$)
$$\begin{aligned}
-\pr_t\MM +\partial_{x_1x_1}\MM -\prx(\pps \MM )&=-\max_{x\in\T^n} \lp W+\frac{|\prx V|}{2}\|\prx V\|_\infty\rp\MM \\
&\quad \leq -\lp W+\frac{\prx V}{2}\pps\rp\MM .\\
\end{aligned}$$
For $\MM $ solution of \eqref{mps} and for $\varepsilon$ sufficiently small, one obtains, by Lemma~\ref{finft}, that
$$\displaystyle\frac{1}{2}\frac{d}{dt}\int_{\T^n}(\tp-\MM )_+^2\leq  \lp \|W\|_\infty+\frac{1}{4\varepsilon}\|\prx V\|_\infty^2 + \frac{\|\prx V\|_\infty^2}{2}\rp\int_{\T^n}\lp(\tp-\MM )_+\rp^2.$$
Then using the fact that $\MM _0\geq \tp_0$ and applying the Gronwall Lemma (since $(\tp-\MM )_+\in L^1((0,T_{\text{max}}),L^2(\T^n))$), one then obtains $\displaystyle\int_{\T^n}(\tp-\MM )_+^2=0$, which means that $(\tp-\MM )_+=0$. 
\end{proof}

\begin{corollary}\label{bou2}
The solution $\p$ given by Theorem~\ref{loc} belongs to $L^2((0,T_{\text{max}}),L^{\infty}(\T^n))$.
\end{corollary}
\begin{proof}
The assertion follows directly from the fact that $\p\leq\MM \,\e^{\frac{-\beta V}{2}}$ (by Lemma~\ref{mp}) and from the embedding $H^1_{x_1}(\T)\hookrightarrow L^{\infty}(\T)$.
\end{proof}

\begin{proposition}
The solution $\p$ given by Theorem~\ref{loc} satisfies:
\begin{equation}\label{stp1}
\p\in L^{\infty}((0,T_{\text{max}}),L^2(\T^n)))\cap L^{2}((0,T_{\text{max}}),H^1(\T^n)))
\end{equation} 
and
\begin{equation}\label{stp2}
\p\in L^{\infty}((0,T_{\text{max}}),H^1(\T^n)))\cap L^{2}((0,T_{\text{max}}),H^2(\T^n))).
\end{equation}
\end{proposition}
\begin{proof}
Recall that by Corollary~\ref{bou2}, one has that $\p\in L^2((0,T_{\text{max}}),L^{\infty}(\T^n))$ and by Remark~\ref{posdi} one has $\op\in C^\infty((0,T_{\text{max}}),C^\infty(\T^n))$.

\textbf{First step:} Multiply \eqref{fpp} by $\p$ and then integrate over space:
$$\displaystyle\frac{1}{2}\frac{d}{dt}\displaystyle\int_{\T^n} |\p|^2+\beta^{-1}\displaystyle\int_{\T^n} |\nabla \p|^2=-\displaystyle\int_{\T^n} \nabla V\p\cdot\nabla\p+\int_{\T^n}\frac{\int_{\T^{n-1}}\pr_{1}V\p}{\int_{\T^{n-1}}\p}\p\pr_{1}\p.$$
Applying Young's inequality, one obtains,
$$\left|\int_{\T^n} \nabla V\p\cdot\nabla\p\right|\leq\|\nabla V\|_{\infty}\lc \frac{\varepsilon}{2}\int_{\T^n}|\nabla\p|^2+\frac{1}{2\varepsilon}\int_{\T^n}|\p|^2\rc. $$
Using the fact that $\p\geq 0$,
$$\left|\int_{\T^n}\frac{\int_{\T^{n-1}}\pr_{1}V\p}{\int_{\T^{n-1}}\p}\p\pr_{1}\p\right|\leq \|\pr_{1}V\|_{\infty}\lc \frac{\varepsilon}{2}\int_{\T^n}|\pr_{1}\p|^2+\frac{1}{2\varepsilon}\int_{\T^n}|\p|^2\rc .$$
Thus, one obtains:
\begin{equation}\label{stp1_0}
\frac{1}{2}\frac{d}{dt}\|\p\|^2_2+(\beta^{-1}-\varepsilon\|\nabla V\|_{\infty})\|\nabla \p\|^2_2\leq \frac{1}{\varepsilon}\|\nabla V\|_{\infty}\|\p\|^2_2.
\end{equation}
Choose $\varepsilon$ such that $\beta^{-1}-\varepsilon\|\nabla V\|_{\infty}>0$. In this case:
$$\frac{1}{2}\frac{d}{dt}\|\p\|^2_2\leq \frac{1}{\varepsilon}\|\nabla V\|_{\infty}\|\p\|^2_2.$$
Using Gronwall Lemma, one gets
$$\|\p\|^2_2\leq \|\p_0\|^2_2\exp\lp \frac{1}{\varepsilon}\|\nabla V\|_{\infty} t\rp ,\quad \forall t\in [0,T_{\text{max}}).$$
Since $\p_0\in L^2(\T^n)$ and $V\in C^2(\T^n)$, then $\p\in L^{\infty}([0,T_{\text{max}}),L^2(\T^n))$. By \eqref{stp1_0}, we have that $\p\in L^{2}([0,T_{\text{max}}),H^1(\T^n))$. Therefore,
\begin{equation}
\p\in L^{\infty}([0,T_{\text{max}}),L^2(\T^n)))\cap L^{2}([0,T_{\text{max}}),H^1(\T^n))).
\end{equation} 
\textbf{Second step:} Multiply \eqref{fpp} by $-\Delta\p$ and then integrate over space:
\begin{equation}\label{steq2}
 \begin{aligned}
&\displaystyle\frac{1}{2}\frac{d}{dt}\displaystyle\int_{\T^n} |\nabla\p|^2+\beta^{-1}\displaystyle\int_{\T^n} |\Delta \p|^2\\
&\qquad \qquad\qquad=-\displaystyle\int_{\T^n}\nabla V\cdot\nabla\p\Delta\p-\int_{\T^n} \Delta V\p\Delta\p+\int_{\T^n}\pr_{1}(\pps \p)\Delta\p\\
&\qquad \qquad\qquad=-\displaystyle\int_{\T^n}\nabla V\cdot\nabla\p\Delta\p-\int_{\T^n}\Delta V\p\Delta\p +\int_{\T^n}\pps \pr_{1}\p\Delta\p\\
&\qquad \qquad\qquad\quad+\displaystyle\int_{\T^n}\frac{\int\pr_{11}V\p}{\int\p}\p\Delta\p+\displaystyle\int_{\T^n}\frac{\int\pr_{1}V\pr_{1}\p}{\int\p}\p\Delta\p\\
&\qquad \qquad\qquad\quad-\displaystyle\int_{\T^n}\frac{\int\pr_{1}V\p}{\int\p}\frac{\int\pr_{1}\p}{\int\p}\p\Delta\p.
\end{aligned}
\end{equation}
Using Young's inequality,
$$\left|\displaystyle\int_{\T^n}\nabla V\cdot\nabla\p\Delta\p\right|\leq\frac{1}{2\varepsilon}\|\nabla V\|_{\infty}^2 \int_{\T^n}|\nabla\p|^2+\frac{\varepsilon}{2}\int_{\T^n}|\Delta\p|^2 $$
and
$$\left|\displaystyle\int_{\T^n}\Delta V\p\Delta\p\right|\leq\frac{1}{2\varepsilon}\|\Delta V\|_{\infty}^2 \int_{\T^n}|\p|^2+\frac{\varepsilon}{2}\int_{\T^n}|\Delta\p|^2 .$$
We have also,
$$\left|\displaystyle\int_{\T^n}\pps \pr_{1}\p\Delta\p\right|\leq \frac{1}{2\varepsilon}\int_{\T^n}|\nabla\p|^2+\frac{\varepsilon}{2}\|\pps \|_{\infty}^2\int_{\T^n}|\Delta\p|^2 .$$
Using the fact that $\p\geq0$,
$$\left|\frac{\int\pr_{11}V\p}{\int\p}\p\Delta\p\right|\leq\frac{1}{2\varepsilon}\|\pr_{11}V\|_{\infty} ^2\int_{\T^n}|\p|^2+\frac{\varepsilon}{2}\int_{\T^n}|\Delta\p|^2 .$$
Similarly, 
$$\left|\displaystyle\int_{\T^n}\frac{\int\pr_{1}V\pr_{1}\p}{\int\p}\p\Delta\p\right|\leq \frac{1}{2\varepsilon}\|\pr_{1}V\|_{\infty}\lp\frac{\|\p\|_\infty}{\min\op}\rp^2\int_{\T^n}|\nabla\p|^2+\frac{\varepsilon}{2}\int_{\T^n}|\Delta\p|^2.$$
And Finally,
$$\left|\displaystyle\int_{\T^n}\frac{\int\pr_{1}V\p}{\int\p}\frac{\int\pr_{1}\p}{\int\p}\p\Delta\p\right|\leq \frac{1}{2\varepsilon}\|\pr_{1}V\|_{\infty}^2\lp\frac{\|\p\|_\infty}{\min\op}\rp^2 \int_{\T^n}|\nabla\p|^2+\frac{\varepsilon}{2}\int_{\T^n}|\Delta\p|^2 .$$
Then taking into account all the previous estimates, one can choose $\varepsilon $ small enough such that Equation \eqref{steq2} can be written as
\begin{equation}\label{stp2_0}
\frac{1}{2}\frac{d}{dt}\|\nabla\p\|^2_{2}+C_{1,\varepsilon}\|\Delta\p\|_{2}^2\leq C_{2,\varepsilon}\lp \frac{\|\p\|_\infty}{\min\op}\rp^2 \|\nabla\p\|^2_{2}+C_{3,\varepsilon}\|\p\|_{2}^2,
\end{equation}
where $C_{i,\varepsilon}>0$, $i=1,2,3$ are three constants.
\medskip

By \eqref{stp1}, one gets
$$\frac{1}{2}\frac{d}{dt}\|\nabla\p\|^2_2\leq C_{2,\varepsilon}\lp \frac{\|\p\|_\infty}{\min\op}\rp^2 \|\nabla\p\|^2_2+C_\varepsilon\|\p_0\|^2_2.$$
The function $C_{2,\varepsilon}\lp \frac{\|\p\|_\infty}{\min\op}\rp^2$ is in $ L^1((0,T_{\text{max}}))$ (by Corollary~\ref{bou2}) and thus, by Gronwall Lemma:
$$\displaystyle \|\nabla\p\|^2_2\leq \|\nabla\p_0\|^2_2\exp\lp \int_0^t C_{2,\varepsilon}\lp \frac{\|\p\|_\infty}{\min\op}\rp^2 dt\rp+C_{4,\varepsilon} ,\quad \forall t\in [0,T_{\text{max}}).$$
 Since $\p_0\in H^1(\T^n)$, then using Equation \eqref{stp2_0}, \eqref{stp1} and elliptic regularity results, one obtains that 
\begin{equation}
\p\in L^{\infty}((0,T_{\text{max}}),H^1(\T^n)))\cap L^{2}((0,T_{\text{max}}),H^2(\T^n))).
\end{equation}
\end{proof}

\begin{lemma}\label{inft}
The solution $\p$ given by Theorem~\ref{loc} belongs to $L^\infty((0,T_{\text{max}}),L^{\infty}(\T^n))$.
\end{lemma}
\begin{proof}
Recall that $\MM \in L^{\infty}((0,T_{\text{max}}),L^2(\T))\cap L^{2}((0,T_{\text{max}}),H^1(\T))$ is the solution of \eqref{mps}. We would like to prove more regularity on $\MM $ (and thus on $\p$). Multiplying Equation \eqref{mps} by $-\partial_{x_1x_1}\MM $, integrating by parts and using Gronwall lemma, one has
$$\begin{aligned}\label{dm}
&\displaystyle\frac{1}{2}\frac{d}{dt}\int_\T|\prx\MM |^2+\beta^{-1}\int_\T|\partial_{x_1x_1}\MM |^2\\
&= -\delta \int_\T|\prx\MM |^2+\int_\T\prx\lp\pps \MM \rp\partial_{x_1x_1}\MM \notag\\
&\leq -\delta \int_\T|\prx\MM |^2+\beta\int_\T|\prx\lp\pps \MM \rp|^2+\frac{\beta^{-1}}{4}\int_\T|\partial_{x_1x_1}\MM |^2\notag\\
&\leq (-\delta+\|\prx V\|_\infty) \int_\T|\prx\MM |^2+\beta\int_\T|\prx\lp\pps\rp|^2 |\MM |^2 +\frac{\beta^{-1}}{4}\int_\T|\partial_{x_1x_1}\MM |^2.
\end{aligned}$$
Recall that 
$$\displaystyle\prx \pps =\displaystyle\frac{\int_{\T^{n-1}}\pr_{x_1x_1}V\p}{\int_{\T^{n-1}}\p}+\frac{\int_{\T^{n-1}}\prx V\prx \p}{\int_{\T^{n-1}}\p} -\frac{\int_{\T^{n-1}}\prx V\p}{\int_{\T^{n-1}}\p}\frac{\int_{\T^{n-1}}\prx \p}{\int_{\T^{n-1}}\p}.
$$
Since $\p\in L^{\infty}((0,T_{\text{max}}),H^1(\T^n)))$, then $\prx \pps\in  L^{\infty}((0,T_{\text{max}}),L^2(\T^n)))$ with
$$\|\prx\pps\|^2_{L^\infty((0,T_{\text{max}}),L^2(\T))}\leq 3\|\pr_{x_1x_1} V\|_\infty^2+6\frac{\|\prx V\|_\infty^2}{(\min \op)^2}\|\prx \p\|_{L^\infty((0,T_{\text{max}}),L^2(\T^n))}^2. $$

Since $\MM \in L^{2}((0,T_{\text{max}}),L^\infty(\T)))$, then
\begin{align*}
\displaystyle\frac{1}{2}\frac{d}{dt}\int_\T|\prx\MM |^2+\frac{3\beta^{-1}}{4}\int_\T|\partial_{x_1x_1}\MM |^2&\leq  (-\delta+\|\prx V\|_\infty) \int_\T|\prx\MM |^2\\
&\quad \displaystyle+\beta\|\prx\pps\|_2^2\|\MM \|_\infty^2.
\end{align*}
Applying now Gronwall lemma, we get
\begin{equation}
\MM \in L^{\infty}((0,T_{\text{max}}),H^1(\T)))\cap L^{2}((0,T_{\text{max}}),H^2(\T))).
\end{equation}
Therefore, the assertion follows directly from the fact that $\p\leq\MM \,\e^{\frac{-\beta V}{2}}$ (by Lemma~\ref{mp}) and from the embedding $H^1(\T)\hookrightarrow L^{\infty}(\T)$.

\end{proof}

The following proposition shows a polynomial bound for the nonlinear functional $ F$ which is useful to prove later the global existence.
\begin{proposition}(Polynomial bound)\label{propol}\\
Let us suppose that $\p :[0,T_{\text{max}})\rightarrow W^{\sigma,p}(\mathbb{T}^n)$ is the solution $\p$ given by Theorem~\ref{loc} of \eqref{fpp}. There exists $C_{F }>0$ such that the nonlinear functional $F $ satisfies the following polynomial bound:
 \begin{equation}\label{pol}
\| F  (\p(t))\|_p\leq C_{F}\|\p(t)\|_{\s,p},\qquad \forall t\in [0,T_{\text{max}}).
\end{equation} 
\end{proposition}

\begin{proof}
For every $ t\in [0,T_{\text{max}})$, using Lemma~\ref{finft},
$$\begin{aligned}
\| F  (\p)\|_p&= \|\nabla V . \nabla \p+\Delta V\p -\prx (\pps \p)\|_p\\
&\leq \|\nabla V\|_{\infty} \|\nabla\p\|_{p}+ \|\Delta V\|_{\infty}\|\p\|_p +\|\pps \|_{\infty}\|\prx \p\|_p\\
&\quad + \displaystyle\left\|\frac{\displaystyle\int_{\T^{n-1}}\!\!\!\!\!\!\pr_{x_1x_1}V\p}{\displaystyle\int_{\T^{n-1}}\!\!\!\p}\p\right\|_p+\left\|\frac{\displaystyle\int_{\T^{n-1}}\!\!\!\!\!\!\prx V\prx \p}{\displaystyle\int_{\T^{n-1}}\!\!\!\p}\p\right\|_p+\left\|\frac{\displaystyle\int_{\T^{n-1}}\!\!\!\!\!\!\prx V\p}{\displaystyle\int_{\T^{n-1}}\!\!\!\p}\frac{\displaystyle\int_{\T^{n-1}}\!\!\!\!\!\!\prx \p}{\displaystyle\int_{\T^{n-1}}\!\!\!\p}\p\right\|_p\\
&\leq\|\nabla V\|_{\infty} \|\nabla\p\|_{p}+ \|\Delta V\|_{\infty}\|\p\|_p +\|\pps \|_{\infty}\|\prx \p\|_p\\
&\quad + \|\pr_{x_1x_1} V\|_{\infty} \|\p \|_{p}+\|\prx  V\|_{\infty}\frac{\|\p\|_\infty}{\min \op}\|\prx  \p\|_p+\|\prx  V\|_{\infty}\|\prx \ln\op\|_{\infty}\|\p\|_p\\
&\leq  \|\nabla V\|_{\infty} \|\p\|_{\s,p}+ \|\Delta V\|_{\infty}\|\p\|_{\s,p} +\|\prx  V\|_{\infty}\| \p\|_{\s,p}+\|\pr_{x_1x_1} V\|_{\infty} \|\p \|_{\s,p}\\
&\quad+\|\prx  V\|_{\infty}\frac{\|\p\|_\infty}{\min \op}\|\p\|_{\s,p}+\|\prx  V\|_{\infty}\|\prx \ln\op\|_{\infty}\|\p\|_p\\
&\leq C(t)\|\p\|_{\s,p},
\end{aligned}$$
where $\displaystyle C(t)=4\| V\|_{C^2}+\|\prx  V\|_{\infty}\frac{\|\p\|_\infty}{\min \op}+\|\prx  V\|_{\infty}\|\prx \ln\op\|_{\infty}$. Now, Lemma~\ref{inft} implies that there exists $C_{F } >0$ such that $C(t)\leq C_F$, $\forall t\in [0,T_{\text{max}})$.
 \end{proof}
 \begin{proposition}\label{prio}
The solution $\p$ given by Theorem~\ref{loc} satisfies
\begin{equation}\label{tau}
\displaystyle \p\in L^{\infty}((0,T_{\text{max}}),W^{\s,p}(\T^n)).
\end{equation}
\end{proposition} 
To prove this proposition, we need the following lemma.
\begin{lemma}\label{gron}
For $t\in [0,T_{\text{max}})$, suppose that there exists a decreasing function $\gamma\in L^1((0,T_{\text{max}}),\mathbb{R}^{+})$ and two functions $u$ and $\zeta\in C([0,T_{\text{max}}),\mathbb{R}^{+})$, such that the following inequality holds
$$u(t)\leq \zeta(t)+\displaystyle\int_0^{t}\gamma (t-s)u(s)ds,\quad \forall t\in [0,T_{\text{max}}). $$
Then, there exists a constant $\delta>0$, depending only on $\gamma$, such that 
$$u(t)\leq 2\zeta^*(t)\exp(\delta t),\quad \mbox{for }\, 0\leq t<T_{\text{max}},$$
where $\zeta^*(t):={\rm max}\{\zeta(s)\,|\, 0\leq s\leq t\}.$
\end{lemma}
The proof of this lemma can be found in \citep{am:84}, Lemma~2.2, see also Appendix~\ref{prgron}.\\
\\
\textbf{Proof of Proposition~\ref{prio}.}
Recall that $\p$ satisfies
\begin{equation}\label{integ}
\p(t)-1=\mathrm{e}^{-t A_p }(\p _0-1)+\displaystyle \int_{0}^{t}\mathrm{e}^{-(t-s) A_p }F (\p(s))ds\qquad \forall t\in (0,T_{\text{max}}).
\end{equation}
Let $t\in (0,T_{\text{max}})$. Using \eqref{or1} and \eqref{pol}, one has\\

$\begin{aligned}
\|\p(t)-1\|_{\sigma,p}\leq & \|{\rm e}^{-tA_p }(\p _0-1)\|_{\sigma,p}+\displaystyle \int_0^{t}\|{\rm e}^{-(t-s)A_p }\|_{\mathcal{L}(L^p,W^{\sigma,p}_0)}\|F (\psi(s))\|_pds\\
\leq & \|{\rm e}^{-tA_p }(\p _0-1)\|_{\sigma,p}+\displaystyle\int_0^tC_\s (t-s)^{-\frac{\sigma}{2}}\e^{-\kappa (t-s)}C_{F}\|\psi\|_{\sigma,p}ds\\
\leq  & \zeta (t)+C_{\s,F}\displaystyle\int_0^{t}(t-s)^{-\frac{\sigma}{2}}\e^{-\kappa (t-s)}\|\psi\|_{\sigma,p}ds,
\end{aligned}$\\
where $\zeta(t):=\|e^{-tA_p }(\p _0-1)\|_{\sigma,p}$. By Lemma~\ref{gron}, there exists $\delta>0$ such that
$$\|\psi(t)\|_{\sigma,p}\leq 2\zeta^*(t)\exp(\delta t),\quad \mbox{for}\, 0\leq t<T_{\text{max}},$$
where $\zeta^*(t):={\rm max}\{\zeta(s)\,|\, 0\leq s\leq t\}\in L^\infty([0,T_{\text{max}}])$. Therefore, the assertion follows immediately.
$\square$

 \subsection{Global existence: Proof of Theorem~\ref{glo}}
The global existence result follows from the following general result.
 \begin{theorem}\label{glrs}
 Suppose that $ F  (\p)\in L^{\infty}((0,T_{\text{max}}),L^p(\T^n))$, then $T_{\text{max}}=+\infty$ and $\p$ is a global solution. 
 \end{theorem}
\begin{proof}
We follow Amann \citep{am:84} to prove this proposition. We suppose that $T_{\text{max}}<+\infty$, then we study the behaviour of $\p(t)$ as $t\rightarrow T_{\text{max}} $. By \eqref{or1}, since $ F  (\p)\in L^{\infty}((0,T_{\text{max}}),L^p(\T^n))$, we have
$$\lp t\mapsto v (t):=\displaystyle\int_0^t {\rm e}^{-(t-s) A_p } F  (\p(s))ds\rp \in C([0,T_{\text{max}}],W^{\s,p}(\T^n)) .$$
Denote $\omega(t):={\rm e}^{-t A_p }\p_0+v(t)$, then $\omega\in C([0,T_{\text{max}}],W^{\s,p})$, since $\p_0\in W^{\s,p}(\T^n).$ Consequently, $\omega(T_{\text{max}})=\displaystyle\lim_{t\rightarrow T_{\text{max}}}\omega(t)$ exists in $W^{\s,p}(\T^n)$. 
\medskip

If $\omega(T_{\text{max}})\in \D ^{\s,p}(\T^n)$, then $\omega$ is a solution of the integral equation \eqref{iee} on $[0,T_{\text{max}}]$ extending $\p$, which contradicts the fact that $\p$ is a maximal solution. Hence, $\omega(T_{\text{max}})\in \lp W^{\s,p}(\T^n)\bigcap \D^{\s,p}(\T^n)\rp^c$, which means that $\overline{\omega} $ is zero at some point, but this is again impossible since $\overline{\omega}=\op>0$. Therefore the assumption $T_{\text{max}}<\infty$ is false, and thus $T_{\text{max}}=+\infty$.\\
\end{proof}

\medskip

We are now in position to prove the global existence result for the initial problem \eqref{fpp} announced in Theorem~\ref{glo}~$(i)$. In fact, using Proposition~\ref{propol} and Proposition~\ref{prio}, we obtain that $F(\p)$ belongs to $L^{\infty}((0,T_{\text{max}}),L^p(\T^n))$. We thus conclude by Theorem~\ref{glrs} that $T_{\text{max}}=+\infty$.

\medskip

This last assertion implies that the orbit $\gamma^+(\p_0)$ exists for all time and is bounded in $W^{s,p}(\T^n)$ for every $s\in [\s,2) .$ Indeed, recall that (by Equation \eqref{or1}) there exists a constant $\kappa>0$ and a constant $C_{\rho}>0$ such that
\begin{equation}\label{or11}
\|\mathrm{e}^{-t A_p }\|_{\mathcal{L}(L^p,W^{\rho,p})}\leq C_{\rho}t^{-\frac{\rho}{2}}\mathrm{e}^{-\kappa t},
\end{equation}
for all $t>0$, provided $0<\rho<2$. Let $0<\delta<T_{\text{max}}$, we can write:
\begin{equation}\label{or2}
\p(t)-1=\mathrm{e}^{-(t-\delta) A_p }e^{-\delta  A_p }(\p_0-1)+\displaystyle \int_0^{t}\mathrm{e}^{-(t-s) A_p } F  (\p(s))ds,\quad \text{for } \delta\leq t<T_{\text{max}}.
\end{equation}
 Since $\mathrm{e}^{-\delta  A_p }(\p_0-1)\in W^{2,p}\hookrightarrow W^{\rho,p}$, then using \eqref{or11}, \eqref{or2} and the boundedness of $ F    (\gamma^+(\p_0))$ in $L^p(\T^n)$, we have
$$\|\p(t)\|_{\rho,p}\leq C_{\rho},\quad \forall\,t\in [\delta,T_{\text{max}}).$$
Since $\T^n$ is bounded, choosing $\rho>\s$ and using the compact embedding $W^{\rho,p}(\T^n)\hookrightarrow W^{\s,p}(\T^n)$, we obtain that $\{\p(t);\delta<t<T_{\text{max}}\}$ is relatively compact in $W^{\s,p}(\T^n)$.

\medskip

The assertion $\gamma^{+}(\p_0)=\{\p(t);\delta<t<T_{\text{max}}\}$ is relatively compact in $C^1(\T^n)$ follows from the fact that $\{\p(t) ; 0 \leq t \leq \delta\}=\p([0,\delta])$ is compact (continuous image of a compact set) and by using the compact embedding \eqref{boun} since $\s>1+\frac{n}{p}$. This proves the assertion~$(ii)$ of Theorem~\ref{glo}.

\appendix
\section{Proofs of various results}
\subsection{Proof of Lemma~\ref{spa}}\label{prspa}

The operator $ A_p$ is a strongly elliptic operator of order $2$ and $D( A_p )=W^{2,p}(\T^n)\cap L^p_0(\T^n)$ which is dense in $L^p_0(\T^n)$ by the embedding \eqref{sobo}. Therefore, by a straightforward adaptation of Theorem~7.3.6 in \citep{paz:83}, where Dirichlet boundary conditions are considered instead of periodic boundary conditions, $- A_p $ generates a strongly continuous analytic semigroup of contraction $\{e^{-t A_p };\, t\geq 0\}$ on $L^p_0(\T^n)$. The Hille-Yosida theorem (see Theorem~1.3.1 in \citep{paz:83}), provides the fact that the resolvent of $- A_p $ contains $\R_+$.

\medskip

 Since the operator $- A_p$ generates a strongly continuous analytic semigroup $\{e^{-t A_p };\, t\geq 0\}$ on $L^p_0(\T^n)$, then $0$ belongs to the resolvent of $- A_p $. Therefore, by Theorem~2.6.13 in \citep{paz:83}, one has the assertions \eqref{reg1} and \eqref{reg2}.
 
\medskip

To prove item~5, we use \eqref{reg1}. Indeed, there exists $\kappa>0$ and $C_\s>0$ such that $\forall \p_0\in L^p_0(\T^n) $
$$\|\mathrm{e}^{-tA_p }\p_0\|_{\s,p}=\| A_p ^{\s/2}\mathrm{e}^{-tA_p }\p_0\|_{p}\leq C_\s t^{-\frac{\sigma}{2}}e^{-\kappa t}\| \p_0\|_{p}.$$
For the last assertion, $\forall \gamma\in [\s/2,1]$ 

$\begin{aligned}
&\left\|\e^{-(t-s) A_p }-\e^{-(r-s) A_p }\right\|_{\mathcal{L}(L^p_0,W^{\sigma,p}_0)}\\
& \qquad\qquad\qquad =\left\| A_p ^{\s/2}(\e^{-(t-r) A_p }-I)\e^{-(r-s) A_p }\right\|_{\mathcal{L}(L^p_0,L^p_0)}\\
&\qquad\qquad\qquad=\left\| A_p ^{\s/2-\gamma}(\e^{-(t-r) A_p }-I) A_p ^{\gamma}\e^{-(r-s) A_p }\right\|_{\mathcal{L}(L^p_0,L^p_0)}\\
&\qquad\qquad\qquad\leq \left\|(\e^{-(t-r) A_p }-I) A_p ^{\s/2-\gamma}\right\|_{\mathcal{L}(L^p_0,L^p_0)}\left\| A_p ^{\gamma}\e^{-(r-s) A_p }\right\|_{\mathcal{L}(L^p_0,L^p_0)}\\
&\qquad\qquad\qquad\leq C_{\s,\gamma}(t-r)^{\gamma-\s/2}(r-s)^{-\gamma}\e^{-\kappa (r-s)}
\end{aligned}$\\

where we used \eqref{reg1} and the fact that, $\forall\p_0\in L^p_0(\T^n) $\\

$\begin{aligned}
\left\|(\e^{-(t-r) A_p }-I) A_p ^{\s/2-\gamma}\p_0\right\|_p&\leq C_{\s,\gamma}(t-r)^{\gamma-\s/2}\left\| A_p ^{\gamma-\s/2} A_p ^{-(\gamma-\s/2)}\p_0\right\|_p\\
&\leq C_{\s,\gamma}(t-r)^{\gamma-\s/2}\|\p_0\|_p
\end{aligned}$\\

where the first inequality is obtained by \eqref{reg2}.

\subsection{Proof of Lemma~\ref{spa2}}\label{prspa2}

Part $(1)$ follows directly from the continuity of $t\mapsto\e^{-t A_p }\p_0   $. To prove \eqref{int2}, let $\p_0  \in L^p_0(\T^n)$ and $h>0$. Then

$\begin{aligned}
\displaystyle\frac{\e^{-h A_p }-I}{h}\int_0^t\e^{-s A_p }\p_0  ds&=\frac{1}{h}\int_0^t\lp \e^{-(s+h) A_p }\p_0  -\e^{-s A_p }\p_0  \rp ds\\
&=\frac{1}{h}\int_h^{t+h}\e^{-s A_p }\p_0  ds-\frac{1}{h}\int_0^t\e^{-s A_p }\p_0  ds\\
&=\frac{1}{h}\int_t^{t+h}\e^{-s A_p }\p_0  ds-\frac{1}{h}\int_0^h\e^{-s A_p }\p_0  ds
\end{aligned}$

and as $h\rightarrow 0$, by \eqref{int1}, the right-hand side tends to $\e^{-t A_p }\p_0  -\p_0  $, which proves \eqref{int2}. For the last assertion, using Theorem~1.2.4(c) in \citep{paz:83} 
$$\displaystyle \frac{d}{dt}\e^{-t A_p }\p_0  =- A_p \e^{-t A_p }\p_0  $$
then \eqref{int3} follows by integrating the last equation from $s$ to $t$.
\subsection{Proof of Proposition~\ref{loc0}}\label{prloc0}
The proof of this result is inspired from~\citep{am:84}, Proposition~2.1. Since $\p _0\in\mathcal{D}^{\sigma,p}(\mathbb{T}^n)$ and $F$ is locally Lipschitz continuous from $\D^{\s,p}(\mathbb{T}^n)$ into $L^p_0(\mathbb{T}^n)$, then there exist $r>0$ and $\lambda \geq 0$ such that
\begin{equation}\label{bl}
\overline{B}_{\sigma,p}(\p _0,2r)\subset \mathcal{D}^{\sigma,p}(\mathbb{T}^n),
\end{equation}
\begin{equation}\label{f}
 \forall \p ,\varphi\in \overline{B}_{\sigma,p}(\p _0,2r),\quad\|F (\p )-F (\varphi)\|_p\leq \lambda \|\p -\varphi\|_{\sigma,p}
\end{equation}
and
\begin{equation}\label{mcst}
M:=\sup \left\{\|F (\p )\|_p,\p \in \overline{B}_{\sigma,p}(\p _0,2r)\right\}<\infty.
\end{equation}
Choose $\delta>0$ such that 
\begin{equation}\label{exp}
 \forall  t\in [0, \delta],\quad\|\mathrm{e}^{-t A_p }(\p _0-1)-(\p _0-1)\|_p\leq r
\end{equation}
and
\begin{equation}\label{alp}
\displaystyle\int_0^{\delta}\widehat{\alpha}(s)ds\leq \min\left\{\frac{1}{2\lambda},\frac{r}{M}\right\},
\end{equation}
where $\widehat{\alpha}$ is defined in Lemma~\ref{spa}. Let
$$Z:=\left\{\!\p \text{ s.t }\p-1 \!\in C([0,\delta],W^{\sigma,p}_0(\mathbb{T}^n));\sup_{0\leq t\leq \delta }\!\|(\p (t)-1)-\mathrm{e}^{-t A_p }(\p _0-1)\|_{\sigma,p}\leq r \!\right\}.$$
$Z$ is endowed with the norm $\|\cdot\|_Z:=\|\cdot\|_{L^{\infty}((0,\delta),W^{\s,p})}$ and it is a complete subset of the Banach space $C([0,\delta],W^{\sigma,p}(\mathbb{T}^n))$. By \eqref{exp}, one has
\begin{equation}\label{z}
\forall \p\in Z,\,\forall t\in [0,\delta],\,\p (t)\in \overline{B}_{\sigma,p}(\p _0,2r).
\end{equation}
Indeed, let $ \p \in Z$, by \eqref{exp}, one has\\

$\begin{aligned}
\|\p (t)-\p _0\|_{\sigma,p}&=\|(\p (t)-1)-(\p _0-1)\|_{\sigma,p}\\
&\leq \|(\p (t)-1)-\mathrm{e}^{-t A_p }(\p _0-1)\|_{\sigma,p}+\|\mathrm{e}^{-t A_p }(\p _0-1)-(\p _0-1)\|_{\sigma,p}\\
&\leq r+r.
\end{aligned}$\\

Since $\forall \p \in Z$, $\overline{B}_{\sigma,p}(\p _0,2r)\subset \D^{\s,p}(\T^n)$, then
\begin{equation}\label{Fpsi}
F (\p (\cdot))\in C([0,\delta],L^p_0)\subset L^{\infty}((0,\delta),L^p_0(\mathbb{T}^n)).
\end{equation}
In addition, by \eqref{mcst}, $\forall \p-1\in Z$ and $\forall t\in [0,\delta]$, one obtains
\begin{equation}\label{Fpsit}
\|F (\p (\cdot))\|_{L^{\infty}((0,\delta),L^p_0)}\leq M.
\end{equation}
Since $- A_p $ generates a strongly continuous analytic semigroup on $L^p_0(\T^n)$ (see Lemma~\ref{spa}), then 
\begin{equation}\label{semigrcon}
t\mapsto \mathrm{e}^{-t A_p }(\p _0-1) \in C([0,\delta],W^{\sigma,p}_0(\mathbb{T}^n)).
\end{equation}
Define now the application $g:\p\in Z\mapsto g(\p)$, where
$$\forall t\in [0,\delta],\,g(\p)(t):=1+\mathrm{e}^{-t A_p }(\p _0-1)+\displaystyle \int_{0}^{t}\mathrm{e}^{-(t-s) A_p }F (\p  (s))ds.$$
We have that $g(Z)\subset Z$. Indeed, using \eqref{or1}, \eqref{Fpsit} and \eqref{alp},  $\forall \p \in Z$, $\forall t\in [0,\delta]$\\

$\begin{aligned}
\|g(\p)(t)-1-\mathrm{e}^{-t A_p }(\p _0-1)\|_{\sigma,p}&=\left\|\displaystyle \int_{0}^{t}\mathrm{e}^{-(t-s) A_p }F (\p  (s))ds\right\|_{\sigma,p}\\
&\leq\displaystyle \int_{0}^{t}\left\|\mathrm{e}^{-(t-s) A_p }F (\p  (s))ds\right\|_{\sigma,p}\\
&\leq M\int_{0}^{t}\widehat{\alpha}(t-s)ds\\
&\leq M \frac{r}{M}.
\end{aligned}$\\

Moreover, by \eqref{f}, \eqref{alp} and \eqref{z}, then $\forall t\in [0,\delta]$, $\forall\, \p_1,\p_2\in Z $
\begin{align*}
\|g(\p _1)-g(\p _2)\|_Z&=\left\|\displaystyle \int_{0}^{t}\mathrm{e}^{-(t-s) A_p }\lp F (\p _1(s))-F (\p _2(s))\right )ds\right\|_Z\\
&\leq \|F (\p _1(\cdot))-F (\p _2(\cdot))\|_{L^{\infty}((0,\delta),L^p_0)}\displaystyle \int_{0}^{t}\widehat{\alpha}(t-s)ds\\
&\leq \lambda \|\p _1-\p _2\|_{\s,p}\frac{1}{2\lambda}\\
&=\frac{1}{2} \|\p _1-\p _2\|_Z.
\end{align*}
In conclusion, $g:Z\rightarrow Z$ is a contraction and the assertion follows by the Banach fixed point theorem.
\subsection{Proof of Theorem~\ref{mldexs}}\label{prmldexs}

Proposition \ref{loc0} implies the existence of a unique mild solution $\p $ on some compact interval $[0,t_1]\subset [0,\infty)$, $t_1>0$, with initial boundary condition $\p _0$ at time $t=0$. If $t_1<\infty$, we can apply Proposition~\ref{loc0} to find a unique mild solution $v$ on $[t_1,t_2]$, for some $t_2>t_1$, with initial boundary condition $\p _1:=\p (t_1)=\mathrm{e}^{\beta^{-1}t_1 \Delta }\p _0+\displaystyle \int_{0}^{t_1}\mathrm{e}^{\beta^{-1}(t_1-s) \Delta }F (\p  (s))ds$, which belongs to $\D^{\s,p}(\T^n)$. Let $w\in C([0,t_2],\mathcal{D}^{\sigma,p}(\mathbb{T}^n))$ be defined as:
$$w=\left\{
\begin{array}{lcl}
\p  & & \mbox{in }[0,t_1],\\
v & & \mbox{in }[t_1,t_2].\\
\end{array}
\right.$$
Then, $w$ is a mild solution on $[0,t_2]$. Indeed, $\forall  t\in [0,t_1]$
$$\p(t)= \mathrm{e}^{\beta^{-1}t \Delta}\p _0+\displaystyle \int_{0}^{t}\mathrm{e}^{(t-s) \beta^{-1}\Delta}F (\p  (s))ds.$$
Now, $\forall  t\in [t_1,t_2]$
\begin{align}
&\mathrm{e}^{\beta^{-1}t \Delta}\p _0+\displaystyle \int_{0}^{t}\mathrm{e}^{(t-s) \beta^{-1}\Delta}F (\p  (s))ds\notag\\
&\quad\quad\quad =\mathrm{e}^{\beta^{-1}t \Delta}\p _0+\displaystyle \int_{0}^{t_1}\mathrm{e}^{\beta^{-1}(t-s)\Delta}F (\p  (s))ds+\displaystyle \int_{t_1}^{t}\mathrm{e}^{\beta^{-1}(t-s) \Delta}F (\p  (s))ds\notag\\
&\quad\quad\quad=\mathrm{e}^{\beta^{-1}(t-t_1) \Delta }\p_1+\displaystyle \int_{t_1}^{t}\mathrm{e}^{\beta^{-1}(t-s) \Delta}F (\p  (s))ds\notag\\
&\quad\quad\quad=v(t),\notag 
\end{align}

since
\begin{align}
&\mathrm{e}^{\beta^{-1}t \Delta}\p _0+\displaystyle \int_{0}^{t_1}\mathrm{e}^{\beta^{-1}(t-s) \Delta }F (\p  (s))ds\notag\\
&\quad\quad\quad=\mathrm{e}^{\beta^{-1}(t-t_1) \Delta }\mathrm{e}^{\beta^{-1}t_1 \Delta }\p _0+\displaystyle \int_{0}^{t_1}\mathrm{e}^{\beta^{-1}(t-s) \Delta }F (\p  (s))ds\notag\\
&\quad\quad\quad=\mathrm{e}^{\beta^{-1}(t-t_1) \Delta }\left [\mathrm{e}^{\beta^{-1}t_1 \Delta }\p _0+\int_{0}^{t_1}\mathrm{e}^{\beta^{-1}(t_1-s) \Delta }F (\p  (s))ds\right ]\notag\\
&\quad\quad\quad=\mathrm{e}^{\beta^{-1}(t-t_1) \Delta }\p_1.\notag
\end{align}

 By Proposition \ref{loc0}, it is also the unique solution on $[0,t_2]$. Define now,
$$J_{\p _0}:=\displaystyle\bigcup \left\{[0,t]\subset[0,\infty) \mbox{ such that } \eqref{fpp} \mbox{ has a unique mild solution on } [0,t]\right\}.$$
$J_{\p _0}$ is an interval in $[0,\infty)$, which contains $0$ and is right open in $[0,\infty)$ since otherwise, an application of Proposition \ref{loc0} to its endpoint would give contradiction. Clearly $J_{\p _0}$ is the maximal interval of existence of a solution $\p $ of $\eqref{fp2}$, which is uniquely defined.

\subsection{Proof of Lemma~\ref{lipsh1}}\label{prlipsh1}
To prove Lemma~\ref{lipsh1}, we will use the following lemma.
\begin{lemma}
Let $r,t>0$ such that $r\leq \frac{t}{2}$, then 
\begin{equation}\label{rhor}
\forall \rho\in (0,1),\quad t^{\rho}+r^{\rho}\leq 3(t-r)^{\rho}.
\end{equation}
\end{lemma}
\begin{proof}
Let $r\leq \frac{t}{2}$, then
$$3(t-r)^\rho\geq 3\lp t-\frac{t}{2}\rp ^\rho=3\lp \frac{t}{2}\rp ^\rho\geq (1+2^\rho)\lp \frac{t}{2}\rp ^\rho=\lp \frac{t}{2}\rp ^\rho+t^\rho \geq r^{\rho}+t^{\rho}. $$
\end{proof}

The proof of the following Lemma is inspired from \citep{am:84}, Proposition~1.4. Since $\p\in C([0,T_{\text{max}}),\D^{\s,p}(\T^n))$ it follows from Lemma~\ref{lips} that $ F  (\p(\cdot))$ $\in C([0,T_{\text{max}}),L^p_0(\T^n))$. Let $0\leq \rho<r\leq t<T <T_{\text{max}}$, then $\displaystyle M:=\sup_{t\in[0,T]}\| F  (\p(t))\|_p<\infty$. Let us now consider two cases:
\begin{itemize}
\item[$(i)$] \label{un} If $0\leq r\leq t/2$, then using \eqref{or1} and \eqref{rhor}\\

$\begin{aligned}
&\displaystyle \left\|\int_0^t\e^{-(t-s) A_p } F  (\p(s))ds-\int_0^r\e^{-(r-s) A_p } F  (\p(s))ds\right\|_{\s,p}\\
&\leq \lc \int_0^t\widehat{\alpha}(t-s)ds +\int_0^r\widehat{\alpha}(r-s)ds\rc \|F (\p (\cdot))\|_{L^{\infty}((0,T),L^p_0)}\\
&\leq C_{\s}M\lc \int_0^t(t-s)^{-\s/2}\e^{-\kappa(t-s)}ds +\int_0^r(r-s)^{-\s/2}\e^{-\kappa(r-s)}ds \rc \\
&\leq C_{\s}M\lp \frac{1}{1-\s/2}t^{1-\s/2}+\frac{1}{1-\s/2}r^{1-\s/2}\rp \\
&\leq\frac{3C_{\s}M}{1-\s/2}(t-r)^{1-\s/2}.
\end{aligned}$
\item [$(ii)$]If $0<t/2\leq r\leq t<T$, thus using $(i)$ and \eqref{dex}, $\forall \gamma\in (\s/2,1)$, one has\\

$\begin{aligned}
&\displaystyle \left\|\int_0^t\e^{-(t-s) A_p } F  (\p(s))ds-\int_0^r\e^{-(r-s) A_p } F  (\p(s))ds\right\|_{\s,p}\\
&\leq \displaystyle \left\|\int_{2r-t}^t\e^{-(t-s) A_p } F  (\p(s))ds-\int_{2r-t}^r\e^{-(r-s) A_p } F  (\p(s))ds\right\|_{\s,p}\\
&\quad +\displaystyle \left\|\int_0^{2r-t}(\e^{-(t-s) A_p }-\e^{-(r-s) A_p }) F  (\p(s))ds\right\|_{\s,p}\\
 &\leq \displaystyle \left\|\int_0^{2(t-r)}\e^{-(2(t-r)-s') A_p } F  (\p(s'))ds'-\int_0^{2r-t}\e^{-((t-r)-s') A_p } F  (\p(s'))ds'\right\|_{\s,p}\\
 &\quad +\displaystyle \left\|\int_0^{2r-t}(\e^{-(t-s') A_p }-\e^{-(r-s') A_p }) F  (\p(s'))ds'\right\|_{\s,p}\\
&\leq \left (\frac{3C_\s M}{1-\s/2} +\frac{C_{\s,\gamma}NMT^{1-\gamma}}{1-\gamma} \rp (t-r)^{\gamma-\s/2}
\end{aligned}$,\\

where we used in the fourth line the following change of variable: $s'=s-2r+t$. Since the result is valid for all $\gamma\in (\s/2,1)$, then
$$\p \in C_{\text{loc}}^{\lp1-\frac{\s}{2}\rp^-}([0,T_{\text{max}}),\D^{\s,p}(\T^n)).$$
\end{itemize}

\subsection{Proof of Lemma~\ref{v1lp}}\label{prv1lp}
First of all, let's prove that $v_1(t)\in D( A_p )$. We define,
\begin{equation}\label{v1ep}
v_{1,\varepsilon}(t)=\left\{
\begin{array}{llll}
\displaystyle\int_0^{t-\varepsilon}\e^{-(t-s) A_p }( F  (\p(s))- F  (\p(t)))ds & \mbox{ for }t\geq \varepsilon,\\
0 & \mbox{ for }t< \varepsilon .
\end{array}
\right.
\end{equation}
From this definition, it is clear that $v_{1,\varepsilon}(t)\rightarrow v_1(t)$ as $\varepsilon \rightarrow 0$ in $L^p$. It is also clear that $v_{1,\varepsilon}(t)\in D( A_p )$ and for $t\geq \varepsilon$ 
\begin{equation}
 A_p v_{1,\varepsilon}(t)=\displaystyle\int_0^{t-\varepsilon} A_p \e^{-(t-s) A_p }( F  (\p(s))- F  (\p(t)))ds.
\end{equation}
Since, by Lemma~\ref{holfp}, $ F  (\p (\cdot))\in C^{\nu^-}_{\text{loc}}([0,T_{\text{max}}),L^p(\T^n))$, it follows that, for $t>0$, $ A_p v_{1,\varepsilon}(t)$ converges as $\varepsilon\rightarrow 0$ and
$$\displaystyle\lim_{\varepsilon\rightarrow 0} A_p v_{1,\varepsilon}(t)=\int_0^{t} A_p \e^{-(t-s) A_p }( F  (\p(s))- F  (\p(t)))ds. $$
The closedness of $ A_p $ then implies that $v_1(t)\in D( A_p )$ for $t>0$ and
\begin{equation}
 A_p v_{1}(t)=\int_0^{t} A_p \e^{-(t-s) A_p }( F  (\p(s))- F  (\p(t)))ds. 
\end{equation}
Now we have only to prove the H\"older continuity of $ A_p v_1(t)$. Since $\e^{-t A_p }$ is a contraction semigroup on $L^p_0(\T^n)$, then $\forall t\in[0,T]$
 \begin{equation}\label{contrs}
\|\e^{-t A_p }\|_{\mathcal{L}(L^p_0,L^p_0)}\leq 1
 \end{equation}
 and
  \begin{equation}\label{contrs2}
\|\e^{-(t+h) A_p }-\e^{-t A_p }\|_{\mathcal{L}(L^p_0,L^p_0)}\leq 2.
 \end{equation}
Using Equation \eqref{reg1} with $\alpha=1$, there exists a constant $C>0$ such that \begin{equation}\label{anls}
\forall t\in (0,T),\quad \| A_p \e^{-t A_p }\|_{\mathcal{L}(L^p_0,L^p_0)}\leq C t^{-1}.
\end{equation}

  Using \eqref{int3}, \eqref{contrs}, \eqref{anls} and \eqref{reg1} (with $\alpha=2$), then for all $0<s<t\leq T< T_{\text{max}}$, we have
  
\begin{align}
  \| A_p \e^{-t A_p }- A_p \e^{-s A_p }\|_{\mathcal{L}(L^p_0,L^p_0)}&=\displaystyle\left \| A_p \int_s^t- A_p \e^{-\tau A_p }d\tau\right\|_{\mathcal{L}(L^p_0,L^p_0)}\notag\\
 &\leq \displaystyle\left \|\int_s^t A_p ^2\e^{-\tau A_p }d\tau\right\|_{\mathcal{L}(L^p_0,L^p_0)}\notag\\
 &\leq \displaystyle\int_s^t\| A_p ^2\e^{-\tau A_p }\|_{\mathcal{L}(L^p_0,L^p_0)}d\tau\notag\\
 &\leq  C \displaystyle\int_s^t\tau^{-2}d\tau\notag\\
 &= C (s^{-1}-t^{-1})\notag\\
 &=Ct^{-1}s^{-1}(t-s). \label{ap}
\end{align}

Let $t\geq 0$ and $h>0$ then

\begin{align}\label{i+ii+iii}
 A_p v_1(t+h)- A_p v_1(t)&= A_p \displaystyle\int_0^t\lp \e^{-(t+h-s) A_p }-\e^{-(t-s) A_p }\rp ( F  (\p(s))- F  (\p(t)))ds\notag\\
&\quad + A_p \displaystyle\int_0^t\e^{-(t+h-s) A_p }( F  (\p(t)- F  (\p(t+h)))ds\notag\\
&\quad + A_p \displaystyle\int_t^{t+h}\e^{-(t+h-s) A_p }( F  (\p(s)- F  (\p(t+h)))ds\notag\\
&=I_1+I_2+I_3.
\end{align}

We estimate each of the three terms separately. From the H\"older continuity of $ F  $ and from \eqref{ap}, one has $\forall 0<t\leq T<T_{\text{max}} $

\begin{align}\label{i}
\|I_1\|_p&\leq \displaystyle\int_0^t\left\| A_p \e^{-(t+h-s) A_p }- A_p \e^{-(t-s) A_p }\right\|_{\mathcal{L}(L^p,L^p)}\| F  (\p(s))- F  (\p(t))\|_pds\notag\\
&\leq Ch\displaystyle\int_0^t\frac{ds}{(t-s+h)(t-s)^{1-\nu^-}}\notag\\
&\leq Ch^{\nu^-}.
\end{align}

The last inequality is obtained by using several changes of variable as follows:

\begin{align}
\displaystyle\int_0^t\frac{ds}{(t-s+h)(t-s)^{1-\nu^-}}&=\int_0^t\frac{du}{(u+h)u^{1-\nu^-}}\\
&=\frac{1}{1-\nu^-} \displaystyle\int_0^{t^{1-\mu}}\frac{v^{\frac{2\mu-1}{1-\nu^-}}}{v^{\frac{1}{1-\nu^-}}+h}dv\\
&=\frac{1}{h}\displaystyle\int_0^{t^{1-\nu^-}}\frac{v^{\frac{2\nu^--1}{1-\nu^-}}}{(\frac{v}{h^{1-\nu^-}})^{\frac{1}{1-\nu^-}}+1}dv\\
&=\frac{1}{h}\displaystyle\int_0^{(t/h)^{1-\nu^-}}\frac{h^{2\nu^--1}w^{\frac{2\mu-1}{1-\nu^-}}h^{1-\nu^-}}{w^{\frac{1}{1-\nu^-}}+1}dw\\
&\leq h^{\nu^--1}\displaystyle\int_0^{\infty}\frac{w^{\frac{2\mu-1}{1-\mu}}}{w^{\frac{1}{1-\nu^-}}+1}dw\\
&\leq Ch^{\nu^--1}
\end{align}

where we have used respectively the following changes of variable: $u=t-s$, $v=u^{1-\mu}$ and $w=\frac{v}{h^{1-\mu}}$.

To estimate $I_2$, we use \eqref{int2}, that $ F  (\p (\cdot))\in C^{\nu^-}([0,T],L^p(\T^n))$ and \eqref{contrs2},

\begin{align}\label{ii}
\|I_2\|_p&= \left\| A_p \displaystyle\int_h^{t+h}\e^{-s A_p }( F  (\p(t))- F  (\p(t+h))ds \right\|_p\notag\\
&=\left\|\lp \e^{-(t+h) A_p }-\e^{-h A_p }\rp ( F  (\p(t))- F  (\p(t+h))\right\|_p\notag\\
&\leq \left\|\e^{-(t+h) A_p }-\e^{-h A_p }\right\|_{\mathcal{L}(L^p,L^p)}\| F  (\p(t))- F  (\p(t+h)\|_p\notag\\
&\leq 2Ch^{\nu^-}.
\end{align}

Finally, to estimate $I_3$, we use \eqref{anls} and that $ F  (\p (\cdot))\in C^{\nu^-}([0,T],L^p_0(\T^n))$,

\begin{align}\label{iii}
\|I_3\|_p&\leq \displaystyle\int_t^{t+h}\left\| A_p \e^{(t+h-s) A_p }\right\|_{\mathcal{L}(L^p_0,L^p_0)}\| F  (\p(s))- F  (\p(t+h)\|_p\notag\\
&\leq C\displaystyle\int_t^{t+h}(t+h-s)^{\nu^--1}ds\notag\\
&\leq Ch^{\nu^-}.
\end{align}

Combining \eqref{i+ii+iii} with estimates \eqref{i}, \eqref{ii} and \eqref{iii}, one obtains that $ A_p v_1(t)$ is H\"older continuous from $[0,T]$ to $L^p(\T^n)$.

\subsection{Proof of Theorem~\ref{milp}}\label{prmilp}

Let $\delta>0$ and $T\in (0,T_{\text{max}})$, we have that
$$(\p(t)-1)=\e^{-t A_p }(\p_0-1)+\displaystyle\int_0^t\e^{-(t-s) A_p } F  (\p(s))ds=\e^{-t A_p }(\p_0-1)+v(t)$$
and
$$\frac{d}{dt}(\p -1)=- A_p (\p -1)+  F  (\p).$$
Since by \eqref{ap} $ A_p \e^{-t A_p }(\p_0-1)$ is Lipschitz continuous from $[\delta,T_{\text{max}})$ into $L^p_0(\T^n)$, then it suffices to show that $ A_p v(t)\in C^{\nu^-}_{\text{loc}}((0,T_{\text{max}}),L^p(\T^n))$. To this end, one can decompose $v$ into
$$v(t)=v_1(t)+v_2(t)=\displaystyle\int_0^t\e^{-(t-s) A_p }( F  (\p(s))- F  (\p(t)))ds+\int_0^t\e^{-(t-s) A_p }  F  (\p(t))ds.$$
From Lemma~\ref{v1lp}, it follows that $ A_p v_1(t)\in C^{\nu^-}_{\text{loc}}((0,T_{\text{max}}),L^p(\T^n))$, so it remains only to show that $ A_p v_2(t)\in C^{\nu^-}_{\text{loc}}((0,T_{\text{max}}),L^p(\T^n))$. By \eqref{int2}, we have $ A_p v_2(t)=-(\e^{-t A_p }-I) F  (\p(t))$, and since $ F  (\p(t))\in C^{\nu^-}_{\text{loc}}([0,T_{\text{max}}),L^p(\T^n))$, it only remains to prove that $\e^{-t A_p } F  (\p(t))\in C^{\nu^-}_{\text{loc}}((0,T_{\text{max}}),L^p(\T^n))$.

Using \eqref{dex} (with $\s=0$, $\gamma=1$, $t=t+h$, $r=t$ and $s=0$), then $\forall \varphi_0\in L^p_0(\T^n)$, $\forall\delta>0$,  $\forall t\geq \delta$ one obtains:
\begin{equation}\label{prin}
\left\|\e^{-(t+h) A_p }\varphi_0-\e^{-t A_p }\varphi_0\right\|_p\leq C\|\varphi_0\|_ph\delta^{-1}.
\end{equation}
 Let $ \delta\leq t\leq T<T_{\text{max}}$ and $h>0$, then using the fact that $ F  (\p)\in C^{\nu^-}([0,T],L^p_0(\T^n))$ and \eqref{prin} 
$$\begin{aligned}
&\left\|\e^{-(t+h) A_p } F  (\p(t+h))-\e^{-t A_p } F  (\p(t))\right\|_p\notag\\
&\leq \left\|\e^{-(t+h) A_p }\right\|_{\mathcal{L}(L^p_0,L^p_0)}\| F  (\p(t+h))- F  (\p(t))\|_p\\
&\quad +\left\|\e^{-(t+h) A_p }-\e^{-t A_p }\right\|_{\mathcal{L}(L^p_0,L^p_0)}\| F  (\p(t))\|_p\notag\\
&\leq Ch^{\nu^-}+C\delta^{-1}h\|F (\p (\cdot))\|_{L^{\infty}((0,T),L^p)}\notag\\
&\leq Ch^{\nu^-}.
\end{aligned}$$
Thus $A_p(\p-1)\in C^{\nu^-}_{\text{loc}}((0,T_{\text{max}}),L^p(\T^n))$. This completes the proof of the part $(i)$ since $\frac{d}{dt}(\p -1)=- A_p (\p -1)+  F  (\p)\in C^{\nu^-}_{\text{loc}}((0,T_{\text{max}}),L^p(\T^n))$. 

\medskip

To prove $(ii)$, we first note that if $\p_0\in W^{2,p}(\T^n)$, then $ A_p \e^{-t A_p }(\p_0-1)=\e^{-t A_p } A_p (\p_0-1)\in C([0,T_{\text{max}}),L^p_0(\T^n))$. By Lemma~\ref{v1lp}, $ A_p v_1(t)\in C^{\nu^-}_{\text{loc}}([0,T_{\text{max}}),L^p(\T^n))$. We also have $ A_p v_2(t)=-(\e^{-t A_p }-I) F  (\p(t))$. Since $ F  (\p)$ is continuous on $[0,T_{\text{max}})$ with values in $L^p_0(\T^n)$, it only remains to show that $\e^{-t A_p } F  (\p(t))$ is continuous on $[0,T_{\text{max}})$ with values in $L^p_0(\T^n)$. From $(i)$, it is clear that $\e^{-t A_p } F  (\p(t))$ is continuous on $(0,T_{\text{max}})$. Since $F  (\p_0)\in L^p_0(\T^n)$, the continuity at $t=0$ follows directly from,
$$\left\|\e^{-t A_p } F  (\p(t))- F  (\p_0)\right\|_p\leq\left\|\e^{-t A_p } F  (\p_0)- F  (\p_0)\right\|_p +\| F  (\p(t))- F  (\p_0)\|_p$$
and this completes the proof of $(ii)$.

\subsection{Proof of Lemma~\ref{gron}}\label{prgron}

This prove is borrowed from  Lemma~3.3 in \citep{am:84}. Choose $\varepsilon>0$ such that $\displaystyle\int_0^{\varepsilon}\gamma(s)ds\leq \frac{1}{2}$. Thus, for $0\leq r\leq t<T_{\text{max}}$\\

$\begin{aligned}
u(r)&\leq \zeta(r)+\displaystyle\int_0^{r-\varepsilon}\gamma (r-s)u(s)ds+\displaystyle\int_{r-\varepsilon}^{r}\gamma (r-s)u(s)ds\\
&\leq \zeta^*(r)+\gamma(\varepsilon)\displaystyle\int_0^{r-\varepsilon}u(s)ds+u^*(t)\displaystyle\int_0^{\varepsilon}\gamma(s)ds,\\
&\leq \zeta^*(t)+\gamma(\varepsilon)\displaystyle\int_0^{t}u^*(s)ds+\frac{1}{2} u^*(t),
\end{aligned}$\\

where $u^*(t):={\rm max}\{u(s)\,|\, 0\leq s\leq t\}$. Therefore,
$$u^*(t)\leq 2\zeta^*(t)+\delta\displaystyle\int_0^{t}u^*(s)ds,\quad 0\leq t<T_{\text{max}},$$
where $\delta:=2\gamma(\varepsilon)$. Standard Gronwall's lemma then yields the assertion.

\section*{Acknowledgments}
 Houssam Alrachid would like to thank the Ecole des Ponts ParisTech and CNRS Libanais for supporting his PhD thesis. The work of Tony Lelièvre is supported by the European Research Council under the European Union's Seventh Framework Programme (FP/2007-2013) / ERC Grant Agreement number 614492. Raafat Talhouk is partially supported by a research contract of the Lebanese university.

\end{document}